\documentclass[11pt,letterpaper]{amsart}
\usepackage{amssymb}
\usepackage{amsfonts}
\usepackage{eucal}
\usepackage{txfonts}
\usepackage{amsmath}
\usepackage{enumerate}
\usepackage{comment}
\usepackage{hyperref}

\newtheorem{thm}{Theorem}[section]

\newtheorem{lem}[thm]{Lemma}

\theoremstyle{definition}

\theoremstyle{remark}
\newtheorem{rem}[thm]{Remark}
\numberwithin{equation}{section}

\makeatletter
\@namedef{subjclassname@2020}{\textup{2020} Mathematics Subject Classification}
\makeatother

\begin{document}
	\title[]
	{On local rigidity theorems with respect to the scalar curvature
	}
	
	\author{Liang Cheng}

	

	
	\thanks{Liang Cheng's  Research partially supported by
		Natural Science Foundation of China 12171180
	}
	
	\address{School of Mathematics and Statistics, and Key Laboratory of Nonlinear Analysis $\&$ Applications (Ministry of Education), Central  China Normal University, Wuhan, 430079, P.R.China}
	
	\email{chengliang@ccnu.edu.cn }

	\begin{abstract}
		By using the Ricci flow, we study local rigidity theorems  regarding scalar curvature, isoperimetric constant and best constant of $L^2$ logarithmic Sobolev inequality. 
	Precisely,	we prove that	if a metric $g$ on an open set $V$ in an $n$-dimensional Riemannian manifold   
	 satisfies
	$$
	\int_V	R(g) dvol_g \ge  0 \text{\ \ and\ \ } I(V)\ge I(\mathbb{R}^n),
	$$
	or   
	$$
	\int_V	R(g) dvol_g \ge 0 \text{\ \ and\ \ } S(V)\ge S(\mathbb{R}^n),
	$$
	then $g=g_{\mathbb{R}^n}$ on $V$,
		where  $R(g)$ is the scalar curvature of $g$, $\mathbb{R}^n$  is Euclidean space, $ I(V)$ is the isoperimetric constant of $V$ and  $S(V)$ is best constant of $L^2$ logarithmic Sobolev inequality of $V$.
		  Moreover,
		we also obtain the local $\mathbb{R}^n$-rigidity about local Perelman's $\nu$-entropy, and  
		local $\mathbb{S}^n$-rigidity (resp. $\mathbb{H}^n$-rigidity)  theorems regarding the cases concerning  $R(g)\ge n(n-1)$ (resp. $R(g)\ge -n(n-1) $),  weighted isoperimetric constant and  best constant of  weighted $L^2$ logarithmic Sobolev inequality for the weighted metric $\left(\cos{\frac{d_g(p,x)}{2}}\right)^{-4}g$ (resp. $\left(\cosh{\frac{d_g(p,x)}{2}}\right)^{-4}g$).
	\end{abstract}
	\maketitle	
	
	\section{Introduction}
	
The	rigidity properties concerning lower scalar curvature bounds are one of important topics in the study of differential geometry. Indeed, M.Gromov raised the following problem: \textit{ Find verifiable criteria for extremality and rigidity, decide which manifolds admit extremal/rigid metrics and describe particular extremal/rigid manifolds; see Problem C in \cite{Gdozen}.}
\textit{	What are the further examples of extremal/rigid with positive scalar curvature; see P27 in \cite{Gromovlecture}.} 
In this paper, we discuss the local rigidity theorems  regarding scalar curvature, isoperimetric constant and best constant of $L^2$ logarithmic Sobolev inequality.

	The  Bishop-Gromov's volume comparison theorem implies that the following local rigidity theorem: if the geodesic ball $B(p,r_0)$ in an $n-$dimensional manifold satisfying $$Rc\ge 0$$ on $B(p,r_0)$  and
	\begin{equation}\label{volume}
		\frac{\operatorname{Vol}(B(p,r_0))}{r_0^n}\ge \omega_n
	\end{equation}
	for some $r_0$, where $\omega_n$ is volume of the unit Euclidean $n$-ball, then
	$B(p,r_0)$ is flat. 
	Our first theorem demonstrates that under the condition that the integral of scalar curvature is non-negative, if we enhance the requirement in  (\ref{volume}) to ensure that the isoperimetric constant is no less than that of Euclidean space, the local rigidity result still holds.

	\begin{thm}\label{rigidity}
		Let $M$ be an $n$-dimensional manifold. Suppose that for some open  subset $V\subset M$  satisfying
		\begin{equation}\label{positive_scalar}
		\int_V R dvol \ge 0
		\end{equation}	
		and 
		\begin{equation}\label{rigidity_2}
			I(V)\doteq \inf\limits_{\Omega\subset V}\frac{\mathrm{Area}(\partial \Omega)}{\mathrm{Vol}(\Omega)^{\frac{n-1}{n}}}	\ge I(\mathbb{R}^n),
		\end{equation}
		where $R$ is the scalar curvature and $I(\mathbb{R}^n)$ is the isoperimetric constant of $n$-dimensional Euclidean space.
		Then $V$ must be  flat.
	\end{thm}

\begin{rem}
The isoperimetric comparison and rigidity theorems were studied for  the manifolds with Ricci curvature at least $n-1$  in classic L\'{e}vy-Gromov inequality  (cf. \cite{G}) and also by Brendle \cite{Brendle2}   for complete noncompact manifolds with non-negative Ricci curvature; also see
\cite{BBG}\cite{CM}.
\end{rem}

\begin{rem}
Llarull \cite{L2} proved if a metric $g$ on $\mathbb{S}^n$ satisfies
	$$
	R(g)\ge R(g_{\mathbb{S}^n})=n(n-1) \text{\ \ and\ \ } g\ge g_{\mathbb{S}^n},
	$$
	where $g_{\mathbb{S}^n}$ is the stand metric on $n$-sphere,
	then  $g=g_{\mathbb{S}^n}$. This result
	shows that  one can not increase the scalar curvature and enlarge the manifold in all directions simultaneously
	on $\mathbb{S}^n$, so Llarull's theorem can be expressed
	by saying the spheres are length extremal.
	However, Llarull's extremality /rigidity  does not hold for Euclidean space $\mathbb{R}^n$
		 since there are $O(n)$-invariant metrics $g\ge g_{\mathbb{R}^n}$ on $\mathbb{R}^n$ with positive scalar curvature (cf.P152 in \cite{Gdozen}).  In contrast with this, Theorem \ref{rigidity} together with Theorem \ref{rigidity_log_sobolev} imply the 
	following local extremality/rigidity on $\mathbb{R}^n$:
	if a metric $g$ on an open set $V\subset M^n$  satisfies
	$$
\int_V	R(g) dvol_g \ge  \int_V R(g_{\mathbb{R}^n})dol_{\mathbb{R}^n}=0 \text{\ \ and\ \ } I(V)\ge I(\mathbb{R}^n),
	$$
	or   
	$$
	\int_V	R(g) dvol_g \ge \int_V R(g_{\mathbb{R}^n})dol_{\mathbb{R}^n}=0 \text{\ \ and\ \ } S(V)\ge S(\mathbb{R}^n),
	$$
	where $ I(V)$ is the isoperimetric constant of $V$ and  $S(V)$ is the best constant of $L^2$ logarithmic Sobolev inequality,
	then $g=g_{\mathbb{R}^n}$ on $V$.
\end{rem}

	We also study the local rigidity theorem for a more generally condition concerning the local Perelman's $\nu$-entropy. Recall the   $\mathcal{W}$-functional was introduced by Perelman \cite{P1} as
	\begin{align}\label{w_func}
		\mathcal{W}(M, g, \varphi, \tau)=-n-\frac{n}{2} \log (4 \pi \tau)+\int_{M}\left\{\tau\left(R \varphi^2+4|\nabla \varphi|^2\right)-2 \varphi^2 \log \varphi \right\} d vol,
	\end{align}
	where $R$ is the scalar curvature of $g$ and $\tau=T-t$.
	Perelman's $\mathcal{W}$-functional has the very nice property: if $g(t)$ is a complete solution to the Ricci flow with bounded sectional curvature, and
	$\varphi^2(t) $ be solution to  the conjugate heat equation 
	\begin{equation}\label{cjhk}
		\left(-\partial_t-\Delta+R\right) u=0 \text{\ on \ } [0,T],
	\end{equation}
	Then $\mathcal{W}(M, g(t), \varphi(t), T-t)$ is monotonely non-decreasing under the Ricci flow. Moreover, if
	$\varphi^2(t) $ is the heat kernel of conjugate heat equation (\ref{cjhk}), then we have $\mathcal{W}(M, g(t), \varphi(t), T-t)\le 0$ and  $\mathcal{W}(M, g(t), \varphi(t), T-t)=0$ at some time if and only
	if the Ricci flow is static flow on Euclidean space (c.f. \cite{P1}\cite{CTY}). 
	A direct consequence of this property is that if $(M^n,g)$ is a complete Riemannian manifold with bounded curvature
	satisfying
	\begin{equation}\label{complete_nv}
		\boldsymbol{\nu}(M, g, \tau):=\inf\limits_{s\in(0,\tau]}\inf\limits_{\varphi \in \mathcal{S}(M)} \mathcal{W}(M, g, \varphi, \tau)\ge 0
	\end{equation}
	for some $\tau>0$, where $\mathcal{S}(M):=\left\{\varphi \mid \varphi \in W_0^{1,2}(M),  \varphi \geq 0,  \int_{M} \varphi^2 d vol_g=1\right\}$, then
	$M^n$ must be isometric to the Euclidean space (c.f. Theorem 4.9 in \cite{w1}). In a previous work \cite{cheng}, the author removed the  bounded curvature assumption and showed
	that if a complete Riemannian manifold $(M^n,g)$ 
	satisfying $\boldsymbol{\nu}(M, g, \tau)\ge 0$ for some $\tau>0$,  then
	$M^n$ must be isometric to the Euclidean space.
	
	In this paper, we study the rigidity for the case local Perelman's $\nu$-entropy is non-negative for some open subset $V$ in an $n$-dimensional manifold $(M,g)$, i.e.  $\boldsymbol{\nu}(V, g, \tau)\ge 0$, where
	\begin{equation}\label{local_mu}
		\begin{aligned}
			\boldsymbol{\nu}(V, g, \tau):&=\inf\limits_{s\in(0,\tau]}\inf\limits_{\varphi \in \mathcal{S}(V)} \mathcal{W}(V, g, \varphi, \tau)\\
			&=\inf\limits_{s\in(0,\tau]}\inf\limits_{\varphi \in \mathcal{S}(V)} \left\{-n-\frac{n}{2} \log (4 \pi \tau)+\int_{V}\left\{\tau\left(R \varphi^2+4|\nabla \varphi|^2\right)-2 \varphi^2 \log \varphi \right\} d vol\right\},
		\end{aligned}
	\end{equation}
	and
	$\mathcal{S}(V):=\left\{\varphi \mid \varphi \in W_0^{1,2}(V),  \varphi \geq 0,  \int_{V} \varphi^2 d vol=1\right\}$, where $R$ denotes the scalar curvature of $g$. For any open subset $V\subset \mathbb{R}^n$,  $\boldsymbol{\nu}(V, g_{\mathbb{R}^n}, \tau)=0$ (c.f. P278 in \cite{w1}). Notice that $\boldsymbol{\nu}(V, g, \tau)$ is not accurately monotone  under the Ricci flow  for the case  $V$ is just 
	an open subset in $M^n$ (c.f. Theorem 5.2 in \cite{w1}). However, we can still obtain the following local rigidity result.

	\begin{thm}\label{rigidity_mu}
		Let $(M,h_0)$ be an $n$-dimensional manifold. Suppose that for some open subset $V\subset M$   satisfying
		\begin{equation}\label{nu_assumption}
			\boldsymbol{\nu}(V, h_0, \tau_0)\ge 0
		\end{equation}
		for some $\tau_0>0$.
		Then $V$ must be flat.	
	\end{thm}

	It was proved by Bakry, Concordet and Ledoux \cite{BCL} originally and  Ni\cite{Ni2} later by a different proof that if an $n$-dimensional complete Riemannian manifold $(M,g)$ with non-negative Ricci curvature satisfies the  $L^2$-logarithmic Sobolev inequality with the best constant of $\mathbb{R}^n$, then $(M,g)$ must be isometric to $\mathbb{R}^n$. Recently, this rigidity theorem was generalized by Balogh, Krist\'{a}ly and Tripaldi \cite{BKT} to the case that  complete Riemannian manifolds  with non-negative Ricci curvature satisfying the  $L^p$-logarithmic Sobolev inequality with the best constant of $\mathbb{R}^n$ for $p>1$.  As an application to Theorem \ref{rigidity_mu}, we can improve Bakry, Concordet and Ledoux \cite{BCL}  and  Ni\cite{Ni2}'s rigidity result to the local case
	 with non-negative integral scalar curvature. Precisely, we have the following result as an application to 
	Theorem \ref{rigidity_mu}.

	\begin{thm}\label{rigidity_log_sobolev}	
		Let $M$ be an $n$-dimensional manifold. Suppose that for some open subset $V\subset M$   satisfying
		\begin{equation}\label{positive_scalar2}
		\int_V R dvol \ge 0
		\end{equation}	
		and  the $L^2$-logarithmic Sobolev inequality with the best constant of $\mathbb{R}^n$ in small scale, i.e.
		\begin{equation}\label{log_sobolev}
			\int_{V} f^2 \log f^2 d vol \le \int_{V}4\tau |\nabla f|^2 d vol-n-\frac{n}{2} \log (4 \pi \tau),
		\end{equation}
		for all $f \in W_0^{1,2}(V)$, $\int_{V} f^2 d vol=1$, $\tau<\tau_0$ and
		for some $\tau_0>0$, where $R$ is the scalar curvature.
		Then $V$ must be flat.
	\end{thm}
\begin{rem}
The
 $L^2$-logarithmic Sobolev inequality for $\mathbb{R}^n$ says for all  $f \in W_0^{1,2}(\mathbb{R}^n)$, $\int f^2 d vol=1$ 
 \begin{equation}
 		\int f^2 \log f^2 d vol \leq \frac{n}{2} \log \left(\frac{2}{n \pi \mathrm{e}} \int|\nabla f|^2 d vol\right) ,
 \end{equation}
 which  is equivalent to that
	\begin{equation}\label{log_R_n}
		\int f^2 \log f^2 d vol \leq \int4\tau |\nabla f|^2 d vol-n-\frac{n}{2} \log (4 \pi \tau),
	\end{equation}
	for all $\tau>0$ (c.f. Lemma 8.1.7 in \cite{Topping}).  We remark that the condition (\ref{log_sobolev}) is more general than (\ref{log_R_n}) since we only
	assume (\ref{log_sobolev}) holds for in small scale, i.e. (\ref{log_sobolev}) holds  for $\tau<\tau_0$ for some positive constant $\tau_0$.
\end{rem}
\begin{rem}\label{ivlog}
 (\ref{log_sobolev}) holds true if $I(V)	\ge I(\mathbb{R}^n)	$ 	(c.f. Theorem 22.16 in \cite{RFV3} ),
so Theorem \ref{rigidity} is just a corollary of Theorem \ref{rigidity_log_sobolev} .
\end{rem}
	
		We next study the corresponding rigidity theorems for the cases $R\ge -n(n-1)$ or $R\ge n(n-1)$. Notice that the hyperbolic space
	$(\mathbb{H}^n,g_{\mathbb{H}^n})$	is conformal to the Euclidean space $(\mathbb{B}^n, g_{\mathbb{R}^n})$ with $g_{\mathbb{R}^n}=\left(\cosh{\frac{d_{g_{\mathbb{H}^n}}(p,x)}{2}}\right)^{-4}g_{\mathbb{H}^n}$ for stand ball $\mathbb{B}^n\subset \mathbb{R}^n$; and  $(\mathbb{S}^n\backslash N,g_{\mathbb{S}^n})$	is conformal to the Euclidean space $(\mathbb{R}^n, g_{\mathbb{R}^n})$ with $g_{\mathbb{R}^n}=\left(\cos{\frac{d_{g_{\mathbb{S}^n}}(p,x)}{2}}\right)^{-4}g_{\mathbb{S}^n}$. We get the following theorem which is corresponding to Theorem \ref{rigidity} for the case $R\ge -n(n-1)$.

	\begin{thm}\label{rigidity_Iso_Hn}
		Let $(M,g)$ be an $n$-dimensional manifold and $p\in M$. If for some relatively compact geodesic ball $B_g(p,r_0)\subset M$ with its
		scalar curvature of $g$
		satisfying 
		\begin{equation}\label{positive_scalar33}
			R_g(x)\ge -n(n-1)
		\end{equation}	
		for all  $x\in B_g(p,r_0)$, and the isoperimetric constant of
		$B_g(p,r_0)$ for the weighted metric  $h=\left(\cosh{\frac{d_g(p,x)}{2}}\right)^{-4}g$  satisfying 
		\begin{equation}\label{iso_Hn}
			I_h(B_g(p,r_0))\doteq \inf\limits_{\Omega\subset B_g(p,r_0)}\frac{\mathrm{Area_h}(\partial \Omega)}{\mathrm{Vol_h}(\Omega)^{\frac{n-1}{n}}}	\ge I(\mathbb{R}^n),
		\end{equation}
		where $\mathrm{Area_h}$ and $\mathrm{Vol_h}$ denote the area and volume with respect to the metric $h$, $ I(\mathbb{R}^n)$ is  the isoperimetric constant of
		$\mathbb{R}^n$, then $\left(B_g(p,r_0),g\right)$ must be isometric to  the $r_0$-ball in hyperbolic space with the sectional curvature  equals to $-1$.
	\end{thm}
\begin{rem}
	We don't know whether or not $\left(B_g(p,r_0),g\right)$ is isometric to  the $r_0$-ball in sphere   if $R_g(x)\ge n(n-1)$ and 	$I_h(B_g(p,r_0))	\ge I(\mathbb{R}^n)$ for weighted metric $h=\left(\cos{\frac{d_g(p,x)}{2}}\right)^{-4}g$?
	In this case, we can only prove $sec_g=1$ at the point $p$ (see Theorem \ref{rigidity_Sn} (i) and Remark \ref{ivlog}).	
\end{rem}
	
	It is also nature to ask the following problems:  Is there  a local $\mathbb{S}^n$-rigidity (resp. $\mathbb{H}^n$-rigidity)  analogous to Theorem \ref{rigidity_log_sobolev}	under the assumptions that  $R_g\ge n(n-1)$ (resp. $R_g\ge -n(n-1)$)  and the $L^2$-logarithmic Sobolev inequality with the best constant of $\mathbb{R}^n$ is satisfied for the weighted metric $h=\left(\cos{\frac{d_g(p,x)}{2}}\right)^{-4}g$ (resp. $h=\left(\cosh{\frac{d_g(p,x)}{2}}\right)^{-4}g$)? 	
For this problem, we can only prove the sectional curvature of $g$ equals to 1 (resp. -1) at the point $p$ if $R_g\ge n(n-1)$ (resp. $R_g\ge -n(n-1)$) and the $L^2$-logarithmic Sobolev inequality with the best constant of $\mathbb{R}^n$ is satisfied for the weighted metric $h=\left(\cos{\frac{d_g(p,x)}{2}}\right)^{-4}g$ (resp. $h=\left(\cosh{\frac{d_g(p,x)}{2}}\right)^{-4}g$) in a neighborhood of $p$.

\begin{thm}\label{rigidity_Sn}
	Let $(M,g)$ be an $n$-dimensional manifold and $p\in M$. 
	
	(i)If there
	exists a neighborhood $V_p$ of $p$ with $diam_g(V_p)<\pi$ and its scalar curvature of $g$ satisfying 
	\begin{equation}\label{positive_scalar_cuvature}
		R_g(x)\ge n(n-1)
	\end{equation}	
	for all  $x\in V_p$, and
	$V_p$ for the weighted metric  $h=\left(\cos{\frac{d_g(p,x)}{2}}\right)^{-4}g$  satisfying the $L^2$-logarithmic Sobolev inequality with the best constant
	of $\mathbb{R}^n$ in small scale, i.e.	
	\begin{equation}\label{LSI_rigidity_2}
		\int_{V_p} f^2 \log f^2 d vol_{h} \le \int_{V_p}4\tau |\nabla^h f|^2 d vol_{h}-n-\frac{n}{2} \log (4 \pi \tau),
	\end{equation}
	for all $f \in W_0^{1,2}(V_p)$, $\int_{V_p} f^2 d vol_{h}=1$, $\tau<\tau_0$ and
	for some $\tau_0>0$, then $sec_g=1$ at $p$.
	
	(ii) If there
	exists a neighborhood $V_p$ of $p$ and its scalar curvature of $g$ satisfying 
	\begin{equation}\label{positive_scalar_cuvature11}
		R_g(x)\ge -n(n-1)
	\end{equation}	
	for all  $x\in V_p$, and
	$V_p$ for the weighted metric  $h=\left(\cosh{\frac{d_g(p,x)}{2}}\right)^{-4}g$  satisfying the $L^2$-logarithmic Sobolev inequality with the best constant
	of $\mathbb{R}^n$ in small scale in sense of (\ref{LSI_rigidity_2}), then $sec_g=-1$ at $p$.
	
\end{thm}

	Finally, it is worth mentioning  related  rigidity  results  involving scalar curvature.
	The positive mass theorem, first proved by Schoen and Yau \cite{STY}\cite{STY4} for $n\le 7$ and later by Witten \cite{W} using spinors that if $(M^n, g)$ is an asymptotically
	flat manifold with non-negative scalar curvature, $n\le 7$ or $M^n$ is spin, then the ADM mass of $M$ is non-negative and  equal to zero if and only if $M$ is isometric to Euclidean space.
There also exists the analogue of the positive mass theorems for asymptotically hyperbolic manifolds by due to Chru\'{s}ciel and Herzlich \cite{CH}, Chru\'{s}ciel and Nagy \cite{CN}, and Wang \cite{Wangx}.	
	Similar techniques for the positive mass theorem can be used to show that the $n$-dimensional torus $T
^n$
does not admit a metric of positive scalar curvature (cf. \cite{STY1}\cite{STY2}\cite{GL}\cite{GL2}).	
Shi and Tam \cite{ST} proved if $M^3$ has non-negative scalar curvature and the region $\Omega\subset M^3$ such that $\partial \Omega$ is mean convex and has positive Gauss curvature, then the Brown-York mass of $\partial \Omega$ has $m_{BY}(\partial \Omega)\ge 0$ and $m_{BY}(\partial \Omega)=0$ if and only if $\Omega $ is isometric to a region in Euclidean space. They also showed similar result is still true in higher dimensions if  each component of $\partial \Omega$
	can be realized as a strictly convex hypersurface in the Euclidean space and if in addition
	$\Omega$ is spin. 
	  However, the analogous rigidity for the case of $S^n_+$, known as Min-Oo's Conjecture, was disproved by 	Brendle, Marques, and Neves \cite{BMN}. In \cite{MinOo} Min-Oo proved if
	  the metric $g$ on $\mathbb{H}^n$
	  satisfying $R_g\ge -n(n-1)$
	  and $g$ agrees with hyperbolic metric outside a compact metric, then $sec_g\equiv -1$.	  
  In the compact setting there is
	the cover splitting scalar rigidity
	theorem of Bray, Brendle, and Neves \cite{Br3}, the $RP^3$ scalar rigidity theorem
	by Bray, Brendle, Eichmair and Neves \cite{Br2}, and sphere rigidity theorem by Marques and Neves \cite{MN}.
	For more results concerning the rigidity theorems with respect to scalar curvature, one may see \cite{Br1} \cite{BM} \cite{Lichao}\cite{QW}, the survey by Brendle \cite{Brendle}, lecture by Gromov \cite{Gromovlecture} and more references therein.

	The present paper is organized as follows. In section 2 we recall some results which we shall use in the next sections. In section 3 we give the proofs of Theorem \ref{rigidity}. In section 4 we give the proofs of Theorem \ref{rigidity_mu} and Theorem \ref{rigidity_log_sobolev}.
	In section 5 we give the proofs of Theorem \ref{rigidity_Iso_Hn} and Theorem \ref{rigidity_Sn}.
	
	\section{Preliminaries}
	
	In this section we recall some results which we shall use in the next sections. The first of these is a result of Perelman's
	Li-Yau-Hamilton type harnack inequality.
	\begin{thm}\cite{P1}\label{Harnack}
		Suppose $\left(M, g(t))\right|_{0 \leq t \leq T}$ is a complete $n$-dimensional Ricci flow with bounded sectional curvature. Let $u\doteq u(p, T; \cdot,\cdot)$ be the heat kernel of the conjugate heat equation centered at
		$(p,T)$, i.e.
		\begin{equation}\label{con_heat_kernel}
			\begin{cases}
				\square^* u=\left(-\partial_t-\Delta+R\right) u=0\text{\quad on \quad}[0,T], \\
				\lim\limits_{t\nearrow T}u=\delta_{p}.\end{cases}
		\end{equation}
		Define
		$$
		\begin{aligned}
			f &:=-\frac{n}{2} \log \left(4 \pi (T-t)\right)-\log u, \\
			v&:=\left\{(T-t)\left(2 \Delta f-|\nabla f|^2+R\right)+f-n\right\} u.
		\end{aligned}
		$$
		Then we have
		$$
		\left(-\frac{\partial}{\partial t}-\Delta+R\right)\nu\\
		= - 2 (T-t) |Rc+\nabla^2 f-\frac{1}{2(T-t)}g|^2 u
		$$
		and
		$$
		v \leq 0.
		$$
	\end{thm}

	We also need the following heat kernel estimates  due to A.Chau, L.Tam, and C. Yu \cite{CTY}; see Corollary 5.2, Proposition 5.1.
	
	\begin{thm}\cite{CTY}\label{heat_kernel_estimates}
		Let $g(t)$, $t\in [0,T]$, be the complete solution to the Ricci flow on an $n$-dimensional manifold $M$ with
		$\sup\limits_{M\times [0,T]}|\nabla^k Rm (g(t))|\le K$ for $0\le k\le 2$.
		Suppose that $ u(y, s; x,t)$ be the heat kernel of the conjugate heat equation (\ref{con_heat_kernel}).
		There exist positive constants $C_1$, $C_2$, $D_1$ and $D_2$ depending on $n, T$ and $K$ such that
		$$
		\begin{gathered}
			\frac{C_1}{\mathrm{Vol}\left(B(y,\sqrt{t-s})\right)} e^{-\frac{d_{g(s)}^2(x, y)}{D_2(t-s)}} \le u(y, s; x,t) \le  \frac{C_2}{\mathrm{Vol}\left(B(y,\sqrt{t-s})\right)} e^{-\frac{d_{g(s)}^2(x, y)}{D_1(t-s)}}.
		\end{gathered}
		$$
		for any $0<s<t<T$.
	\end{thm}
	
	Finally we need the following gradient estimates for the positive solutions
	to the conjugate heat equations by Lei Ni \cite{Ni} and also see Lemma 6.3 in \cite{CTY}.
	\begin{thm}\cite{Ni}\cite{CTY}\label{gradient_estimates}
		Let $g(t)$, $t\in [0,T]$, be the complete solution to the Ricci flow on an $n$-dimensional manifold $M$ with
		$\sup\limits_{M\times [0,T]}|\nabla^k Rm (g(t))|\le K$ for $0\le k\le 2$.	
		Suppose $u>0$ is a solution of
		$$
		\square^* u=\left(-\partial_t-\Delta+R\right) u=0,
		$$
		and $u \leq A$  on $M\times [0,T]$. Then
		$$
		(T-t) \frac{|\nabla u|^2}{u} \leq C\left[u \log \frac{A}{u}+u\right]
		$$
		for some constant $C$ depending only on $T, n$ and $K$.
	\end{thm}
	
	Given an open domain in a Riemannian manifold where curvature and injectivity radius is controlled, the following  result of Hochard  that allows us to "push the boundary to infinity" by a conformal modification of the metric in a neighborhood of the boundary, while keeping curvature and injectivity radius controlled.
	
	\begin{thm} [Corollaire IV.1.2 in \cite{Ho}]\label{Ho}
		Let $U$ be an open domain in a Riemannian manifold $(M, g)$ such that for every $x \in U$, $B(x, 1)$ is relatively compact in $U$, $$\operatorname{inj}_{\mathrm{x}} g \geq 1$$ and $$|\operatorname{Rm} (g(x))| \leq 1.$$ Then there exists a domain $(U)_1 \subset \tilde{U} \subset U$ such that for any $k>C(n)$ one can produce a Riemannian metric $h_k$ on $\tilde{U}$ with
		$\left(\tilde{U}, h_k\right)$ is a complete Riemannian manifold,
		$$
		h_k \equiv g \text { on }(\tilde{U})_{\frac{C(n)}{\sqrt{k}}} \text {, }
		$$
		$$\left|\operatorname{Rm} (h_k(x))\right| \leq k$$ and $$\operatorname{inj}_{\mathrm{x}} h_k \geq \frac{1}{\sqrt{k}}$$ for $x \in \tilde{U},$
		where	$(U)_r=\{x \in U \mid B(x, r) \text { is relatively compact in } U\}$
	\end{thm}
	\begin{rem}\label{remark}
		With little more observations to Hochard's proofs of Theorem \ref{Ho}, we can also  make the metric $h_k$ constructed in Theorem \ref{Ho} such that for any $m\ge 0$
		$$\left|\nabla^m\operatorname{Rm} (h_k(x))\right| \leq K,$$ where
		$K$ is a constant only depending on $m$, $\sup\limits_{U}|\nabla^m Rm (g)|$.
		Actually Hochard define a metric $h_k$ on $U$, under the form $h_k=e^{2 f \circ \rho} g$,  where $\rho$ is a smooth 'distance-like' function defined in \cite{Ho} and $f:\left[0,1\left[\rightarrow \mathbb{R}_{+}\right.\right.$is a smooth function that has to be chosen
		as
		$$
		f(x)= \begin{cases}0 & \text { for } 0 \leq x \leq 1-\epsilon, \\ -\log \left(1-\left(\frac{x-1+\epsilon}{\epsilon}\right)^2\right) & \text { for } 1-\epsilon<x<1.\end{cases}
		$$
		The the sectional curvature is change as the following under conformal transformation (see for example \cite{Be}, page 58)
		$$
		\mathrm{K}_{\mathrm{h}}(\sigma)=\left(\mathrm{K}_{\mathrm{g}}(\sigma)-\left.\operatorname{tr} \nabla^2(f \circ \rho)\right|_\sigma+\left.|d(f \circ \rho)|_\sigma\right|^2-|d(f \circ \rho)|^2\right) e^{-2 f \circ \rho} .
		$$
		We only need modify $f$ little as
		$$
		f(x)= \begin{cases}0 & \text { for } 0 \leq x \leq 1-\epsilon, \\ -\textbf{m}\log \left(1-\left(\frac{x-1+\epsilon}{\epsilon}\right)^2\right) & \text { for } 1-\epsilon<x<1,\end{cases}
		$$
		and notice that $f^{(m)}e^{-2f}\le c_{m,\epsilon}$, where $c_{m,\epsilon}$ is a constant only depending on $m$ and $\epsilon$. The rest of the proof is
		the same as in \cite{Ho} and we omit the details and leave it to the readers.
	\end{rem}

We need the following lemma which due to Perelman \cite{P1}.
	
	\begin{lem}\label{key_1}
		Let $\left\{\left(M, g(t)\right), 0 \leq t \leq T \right\}$ be a complete Ricci flow solution.
		Let $u(x,t)\doteq u(p, T; x,t)$ be the heat kernel of the conjugate heat equation (\ref{con_heat_kernel}) centered at
		$(p,T)$
		and
		$h:M\to [0,1]$ be a time-independently smooth cut-off function satisfying
		$$
		h(x)= \begin{cases}1, & \forall x \in \Omega_{0}^{\prime} ; \\ 0, & \forall x \in M \backslash \Omega_{0} .\end{cases}
		$$
		with $\Omega_0'\subset \Omega_{0}\subset M$.
		Denote 	$S:=\left.\int_{M} u hd V_t\right|_{t=0} $, $u=\left(4\pi(T-t)\right)^{-\frac{n}{2}}e^{-f}$, $v=\left\{(T-t)\left(2 \Delta f-|\nabla f|^2+R\right)+f-n\right\} u$ and  Perelman's  $\mathcal{W}$ entropy $$ \mathcal{W}(\Omega, g, \varphi, \tau)
		=-n-\frac{n}{2} \log (4 \pi \tau)+\int_{\Omega}\left\{\tau\left(R \varphi^2+4|\nabla \varphi|^2\right)-2 \varphi^2 \log \varphi \right\} d vol_{g},$$
		where $\varphi\in \mathcal{S}(V)=\left\{\varphi \mid \varphi \in W_0^{1,2}(V),  \varphi \geq 0,  \int_{V} \varphi^2 d vol_g=1\right\}$. 
		Then we have
		$$
		\left. \int_{M} v h d V_t \right|_{t=0}
		\ge 	 \mathcal{W}(M, g(0), \tilde{u}^{\frac{1}{2}}(0),T)S-\left. \int_{M\backslash \Omega'_{0}}\left\{4 T|\nabla \sqrt{h}|^2-h \log h\right\} ud V_t\right|_{t=0},
		$$
		where  $\tilde{u}=\frac{u h}{S}$.
		
	\end{lem}
	\begin{proof}
		Noted that  $\left.\int_M \tilde{u} d V_t \right|_{t=0}=1$ and $\tilde{u}$ is supported on $\Omega_{0}$ at $t=0$. Denote
		$\tilde{f} =-\log \tilde{u}-\frac{n}{2} \log \left(4 \pi (T-t)\right)=f-\log h+\log S		$.	We obtain
		\begin{equation*}\label{key2}
			\begin{aligned}
				&\ \ \ \ \mathcal{W}(M, g(0), \tilde{u}^{\frac{1}{2}}(0), T)\\
				&=\left. \int_{M}\left\{T\left(R+2 \Delta \tilde{f}-|\nabla \tilde{f}|^2\right)+\tilde{f}-n\right\} \tilde{u}d V_t \right|_{t=0} \\
				&=\left\{\log S+\frac{1}{S} \left.\int_{M} v h  d V_t \right|_{t=0}\right\} +\left.\frac{1}{S} \int_{M}\left\{T(-2\Delta \log h-|\nabla \log h|^2+2\langle \nabla f,\nabla \log h\rangle)-\log h  \right\} uh d V_t \right|_{t=0} \\
				&=\left\{\log S+\frac{1}{S} \left.\int_{M} v h  d V_t \right|_{t=0}\right\} +\left.\frac{1}{S} \int_{M}\left\{4 T|\nabla \sqrt{h}|^2-h \log h\right\} u d V_t \right|_{t=0} \\
				&\le \frac{1}{S} \left.\int_{M} v h  d V_t \right|_{t=0} + \left.\frac{1}{S} \int_{M\backslash \Omega'_{0}}\left\{4 T|\nabla \sqrt{h}|^2-h \log h\right\} u d V_t\right|_{t=0} ,
			\end{aligned}
		\end{equation*}
		where we used $h\equiv 1$ on $\Omega'_{0}$ and $S=\left.\int_M uhd V_t\right|_{t=0}\le\left. \int_M ud V_t\right|_{t=0} =1$ in the above inequalities.
	\end{proof}

	The following lemma is just a local version of an estimate by R.Bamler \cite{RB3}, we present the proof below for sake of convenience for the readers.
	\begin{thm}\cite{RB3}\label{key_2}
		Let $\left\{\left(M, g(t)\right), 0 \leq t \leq T \right\}$ be a complete Ricci flow solution.
		Let $u(x,t)\doteq u(p, T; x,t)$ be the heat kernel of the conjugate heat equation (\ref{con_heat_kernel}) centered at
		$(p,T)$
		and
		$h:M\to [0,1]$ be a time-independently smooth cut-off function satisfying
		$$
		h(x)= \begin{cases}1, & \forall x \in \Omega_{0}^{\prime} ; \\ 0, & \forall x \in M \backslash \Omega_{0} .\end{cases}
		$$
		with $\Omega_0'\subset \Omega_{0}\subset M$. Then			
		we have for any $b>0$ the following estimate holds
		\begin{equation}\label{key_2_main}
			\begin{aligned}
				\frac{d}{d t} \int_M \tau R h u d V_t \geq & - b  \int_M \tau\left(|\operatorname{Rc}|^2+(2n+1)|\nabla f|^4+2n^2\left|\nabla^2 f\right|^2\right) h u d V_t  \\
				& -3b^{-1}  \int_M \tau\left|\operatorname{Rc}+\nabla^2 f-\frac{1}{2 \tau} g\right|^2 h u d V_t \\
				& -  \int_M \tau^{-1}\left(\tau\left(-|\nabla f|^2+\triangle f\right)+f-\frac{n}{2}\right) \left(f-\frac{n}{2}\right) h u d V_t \\
				& -\tau \int_M\left( 2\langle \nabla R,\nabla h \rangle+R\Delta h\right) u d V_t +2 \tau\int_M \operatorname{Rc}(\nabla f, \nabla h)u d V_t \\
				& -\int_M\left(-\tau\left(|\nabla f|^2+R\right)+f \right)\langle \nabla f, \nabla h \rangle u d V_t,
			\end{aligned}
		\end{equation}
		where $\tau=T-t$ and $	f=-\frac{n}{2} \log \left(4 \pi \tau\right)-\log u$.
	\end{thm}
	\begin{proof}
		We calculate that
		\begin{equation}\label{p_1}
			\begin{aligned}
				&\ \ \ \	\tau  \frac{d}{d t} \int_M \tau R h u d V_t \\
				& =\tau \int_M u\left(\frac{\partial}{\partial t}-\Delta\right)(\tau R h) d V_t\\
				& =\tau \int_M\left(2 \tau|\operatorname{Rc}|^2-R\right) h u d V_t -\tau^2 \int_M\left( 2\langle \nabla R,\nabla h \rangle+R\Delta h\right) u d V_t\\
				&=\tau \int_M\left(2 \tau \operatorname{Rc} \cdot\left(\operatorname{Rc}+\nabla^2 f-\frac{1}{2 \tau} g\right)-2 \tau \operatorname{Rc} \cdot \nabla^2 f\right) h u d V_t -\tau^2 \int_M\left( 2\langle \nabla R,\nabla h \rangle+R\Delta h\right) u d V_t,
			\end{aligned}
		\end{equation}
		where we used evolution equation $\left(\frac{\partial}{\partial t}-\Delta\right) R =2|Rc|^2$ (cf.\cite{H1}).
		
		Then for any $b>0$,
		\begin{equation}\label{est_1}
			\begin{aligned}
				&\ \ \ \	\tau  \frac{d}{d t} \int_M \tau R h u d V_t \\
				& \geq-b \tau^2 \int_M|\operatorname{Rc}|^2 h u d V_t-b^{-1} \tau^2 \int_M\left|\operatorname{Rc}+\nabla^2 f-\frac{1}{2 \tau} g\right|^2 h u d V_t\\
				&\ \ -2 \tau^2 \int_M \left(\operatorname{Rc} \cdot \nabla^2 f\right) h u d V_t-\tau^2 \int_M\left( 2\langle \nabla R,\nabla h \rangle+R\Delta h\right) u d V_t.
			\end{aligned}
		\end{equation}
		The terms
		\begin{equation}\label{est_2}
			\begin{aligned}
				&\ \ \ \	2 \tau^2 \int_M \operatorname{Rc} \cdot \nabla^2 f h u d V_t\\
				= & 2 \tau^2 \int_M\left(-\operatorname{div} \operatorname{Rc} \cdot \nabla f+\operatorname{Rc}(\nabla f, \nabla f)\right) h u d V_t
				-2 \tau^2 \int_M \operatorname{Rc}(\nabla f, \nabla h)u d V_t\\
				= & 2 \tau^2 \int_M\left(\operatorname{Rc}+\nabla^2 f-\frac{1}{2 \tau} g\right) \cdot(\nabla f \otimes \nabla f) h u d V_t \\
				& \quad+\tau^2 \int_M\left(-\nabla R-2 \nabla_{\nabla f} \nabla f+\frac{1}{\tau} \nabla f\right) \cdot \nabla f h u d V_t 	-2 \tau^2 \int_M \operatorname{Rc}(\nabla f, \nabla h)u d V_t\\
				\leq & b \tau^2 \int_M|\nabla f|^4 h u d V_t+b^{-1} \tau^2 \int_M\left|\operatorname{Rc}+\nabla^2 f-\frac{1}{2 \tau} g\right|^2 h u d V_t \\
				& +\tau \int_M \nabla\left(-\tau\left(|\nabla f|^2+R\right)+f \right) \cdot \nabla f h u d V_t
				-2 \tau^2 \int_M \operatorname{Rc}(\nabla f, \nabla h)u d V_t ,
			\end{aligned}
		\end{equation}
		and
		\begin{equation}\label{est_3}
			\begin{aligned}
				&\ \ \ \  \tau \int_M \nabla(-\left.\tau\left(|\nabla f|^2+R\right)+f \right) \cdot \nabla f h u d V_t\\
				&=\tau \int_M\left(-\tau\left(|\nabla f|^2+R\right)+f \right)\left(|\nabla f|^2-\triangle f\right) h u d V_t
				+\tau \int_M\left(-\tau\left(|\nabla f|^2+R\right)+f \right)\langle \nabla f, \nabla h \rangle u d V_t \\
				&=-\tau^2 \int_M\left(\operatorname{Rc}+\nabla^2 f-\frac{1}{2 \tau} g\right) \cdot\left(\left(|\nabla f|^2-\triangle f\right) g\right) h u d V_t\\
				&\ \ \ \  +\tau \int_M\left(\tau\left(-|\nabla f|^2+\triangle f\right)+f-\frac{n}{2}\right)\left(|\nabla f|^2-\Delta f\right) h u d V_t\\
				&\ \ \ \ +\tau \int_M\left(-\tau\left(|\nabla f|^2+R\right)+f \right)\langle \nabla f, \nabla h \rangle u d V_t \\
				& \leq n \tau^2 b \int_M\left(|\nabla f|^2-\triangle f\right)^2 h u d V_t+\tau^2 b^{-1} \int_M\left|\operatorname{Rc}+\nabla^2 f-\frac{1}{2 \tau} g\right|^2 h u d V_t \\
				& \quad-\int_M\left(\tau\left(-|\nabla f|^2+\triangle f\right)+f-\frac{n}{2}\right)^2 h u d V_t  +\int_M\left(\tau\left(-|\nabla f|^2+\triangle f\right)+f-\frac{n}{2}\right)\left(f-\frac{n}{2}\right) h u d V_t\\
				&\quad +\tau \int_M\left(-\tau\left(|\nabla f|^2+R\right)+f \right)\langle \nabla f, \nabla h \rangle u d V_t \\
				& \leq 2 n \tau^2 b \int_M|\nabla f|^4 h u d V_t+2 n^2 \tau^2 b \int_M\left|\nabla^2 f\right|^2 h u d V_t+\tau^2 b^{-1} \int_M\left|\operatorname{Rc}+\nabla^2 f-\frac{1}{2 \tau} g\right|^2 h u d V_t \\
				& \quad  +\int_M\left(\tau\left(-|\nabla f|^2+\triangle f\right)+f-\frac{n}{2}\right)\left(f-\frac{n}{2}\right) h u d V_t +\tau \int_M\left(-\tau\left(|\nabla f|^2+R\right)+f \right)\langle \nabla f, \nabla h \rangle u d V_t \\
			\end{aligned}
		\end{equation}
		Then (\ref{key_2_main})	holds by
		combining (\ref{est_1}), (\ref{est_2}) and (\ref{est_3}).

	\end{proof}	

	\section{Proofs of the Theorem \ref{rigidity}}

	Before presenting our proof of Theorem \ref{rigidity}, we sketch our strategy for the proof of Theorem \ref{rigidity}. 	
With the extra conditions that $(V,h_0)$  is complete and has the bounded sectional curvature, Theorem \ref{rigidity} can be easily obtained by the monotonicity of Perelman's $\mathcal{W}$-functional and the fact $\mathcal{W}\ge 0$ if the conditions of Theorem \ref{rigidity} hold on the hole complete manifold.  However, Perelman's localized entropy $\boldsymbol{\nu}(V, h_0, \tau)$ is not accurately  monotone under the Ricci flow  for the case  $V$ is just 
an open subset in a manifold (c.f. Theorem 5.2 in \cite{w1}).
Our strategy of the proof of the following: for any $p\in V$,
we first do a
conformal change to the initial metric, making it a complete metric with bounded curvature
and leaving it unchanged on smaller open ball $B_{h_0}(p,r_0)\subset V$, and then run a complete Ricci flow $\bar{g}(t)$ up to a
short time by using Shi's classical existence theorem  from \cite{Sh}.
We shall show that  the Ricci curvature for $\bar{g}(t)$  satisfies
\begin{equation}\label{lp_key}
	\int_{B_{h_0}(p,\frac{r_0}{5})}  |\overline{Rc}|^2(x,\xi T) \bar{u}(p, T; x,\xi T)  d \overline{V}_{T\xi}(x) \le  c T^{-2}e^{-\frac{cr_0^2}{1600T}}
\end{equation}
for some $\xi\in [\frac{1}{4},\frac{1}{2}]$,
where $ \bar{u}(p, T; x,t)$ be the heat kernel of the conjugate heat equation (\ref{con_heat_kernel}) centered at $(p,T)$ with respect to $\bar{g}(t)$. 
 This tell us $|\overline{Rc}|(p,T)\to 0$ by letting $T\to 0$ in (\ref{lp_key}) and  $V$ is Ricci flat, so $V$ is flat by the  Bishop-Gromov's volume comparison.	
	Now we give the details of proof of Theorem \ref{rigidity}. 
	\begin{proof}[Proof of Theorem \ref{rigidity}]
Note that (\ref{rigidity_2}) implies that the growth for volume of geodesic $r$-
balls is no less than
 $ \omega_n$ for all $x\in V$ and $r$ sufficient small. It follows from the expansion for volumes of geodesic balls
			$\operatorname{Vol}(B(x, r))=\omega_n r^n\left(1-\frac{R(x)}{6(n+2)} r^2+O\left(r^3\right)\right)$ that scalar curvatures are non-positive for all $x\in V$ if (\ref{rigidity_2}) holds. Then (\ref{positive_scalar}) implies  that  scalar curvatures are zero  for all $x\in V$.
			
		We now take an arbitrary fixed point $p \in V $ and relatively compactly
		small geodesic ball $B_{h_0}(p,2r_0)\subset V$ with its metric
		denoted by $h_0$ satisfying the $R_{h_0}\ge 0$ and (\ref{rigidity_2}).
		Firstly, we applying the conformal change to $B_{h_0}(p,2r_0)$ by Theorem \ref{Ho} and Remark \ref{remark} to get  a complete metric $\bar{g}_0$ on $B_{h_0}(p,2r_0)$ such that
		$$\bar{g}_0=h_0 \text{\ \ on\ \ }B_{h_0}(p,r_0),$$
		\begin{equation}\label{injectivity}
			\operatorname{inj}(\bar{g}_0)\ge \kappa, \text{\ \ \ \qquad \ }
		\end{equation}
		and
		$$|\nabla^k \operatorname{Rm}_{\bar{g}_0}| \leq K$$
		for all $k\ge 0$ throughout $B_{h_0}(p,2r_0)$, where $\kappa$ is a constant only depending on $r_0$, $	\operatorname{inj}(h_0)$ and $\sup\limits_{B_{h_0}(p,2r_0)}|Rm_{h_0}|$ and  $K$ is a constant only depending on $r_0$, $k$ and $\sup\limits_{B_{h_0}(p,2r_0)}|\nabla^kRm_{h_0}|$.
		Then by Shi's local existence theorem \cite{Sh} for the noncompact Ricci flow, we have a complete solution to Ricci flow $\bar{g}(t)$ on $B_{h_0}(p,2r_0)\times [0,\frac{1}{16K}]$ such that
		$$
		\bar{g}(0)=\bar{g}_0,
		$$
		and
		$$|\operatorname{Rm}_{\bar{g}(t)}| \leq 2K$$ throughout $B_{h_0}(p,2r_0)\times [0,\frac{1}{16K}]$.  Together with Shi's estimates \cite{Sh} and modified Shi's interior estimates \cite{LT}, we have
		\begin{equation}\label{curvature_bound}
			|\nabla^k\operatorname{Rm}_{\bar{g}(t)}| \leq C(K)
		\end{equation}	
		throughout $B_{h_0}(p,2r_0)\times [0,\frac{1}{16K}]$ for all $k\ge 0$.
		
		Now for any fixed $0<T\le \min\{\frac{1}{16K},1\}$, we consider the
		rescaled Ricci flow	$g(t)=T^{-1}\bar{g}(Tt)$ for $0\le t\le 1$. 	Let $u(x,t)\doteq u(p, 1; x,t)$ be the heat kernel of the conjugate heat equation (\ref{con_heat_kernel}) centered at
		$(p,1)$ with respect to $g(t)$. Combing with (\ref{injectivity}) and (\ref{curvature_bound}), we get from Theorem \ref{heat_kernel_estimates}
		that
		\begin{equation}\label{heat_kernel_estimates_1}
			\begin{gathered}
				C_1(1-t)^{-\frac{n}{2}} e^{-\frac{d_{g(0)}^2(p, x)}{D_2(1-t)}} \le u( x,t) \le  C_2(1-t)^{-\frac{n}{2}} e^{-\frac{d_{g(0)}^2(p,x)}{D_1(1-t)}}.
			\end{gathered}
		\end{equation}
		for any $0<t<1$ with
		positive constants $C_1$, $C_2$, $D_1$ and $D_2$ depending on $n$ and $K$.
		
		Here and below,  we shall use the notations $c$, $c'$, $c''$ or $c_l$ on different lines if the constants differ at most by a multiplicative factor depending only
		on $n$ and $K$. We also use the notations $\operatorname{Rc}$ and $R$ without the subscripts to denote the Ricci curvature and scalar curvature with respect to $g(t)$.

		Let $\psi$ be a cut-off function such that $\psi \equiv 1$ on $(-\infty, 1)$, $\psi \equiv 0$ on $(2, \infty)$, $0\le\psi\le 1$ and $-10 \leq \psi^{\prime} \leq 0$ everywhere. Moreover, $\psi$ satisfies
		$$
		\psi^{\prime \prime} \geq-10 \psi, \quad\left(\psi^{\prime}\right)^2 \leq 10 \psi.
		$$
		Take $$h_1(x)=\psi\left(\frac{d_{g(0)}(p,x)}{10 A}\right).$$ Let
		$S:=\left.\int_{M} u h_1d V_t\right|_{t=0} $ and $\tilde{u}=\frac{u h_1}{S}$. And $\tilde{u}$ is supported on $B_{h_0}(p,r_0)$ at $t=0$ if $A\le  \frac{1}{20}T^{-\frac{1}{2}}r_0$. It follows from (\ref{rigidity_2}) and  the scalar curvature is zero at $t=0$ on $B_{h_0}(p,r_0)$ that
		\begin{equation}\label{wge0}
			\begin{aligned}
				\mathcal{W}(M, g(0), \tilde{u}^{\frac{1}{2}}(0),1)
				&= \mathcal{W}(B_{h_0}(p,r_0), g(0), \tilde{u}^{\frac{1}{2}}(0),1)\\
				&= -n-\frac{n}{2} \log (4 \pi )+\int_{B_{h_0}(p,r_0)}\left\{4|\nabla \tilde{u}^{\frac{1}{2}}(0)|^2- \tilde{u}(0) \log \tilde{u}(0) \right\} d vol_{g(0)}\\
				&\ge 0
			\end{aligned}
		\end{equation}
		(c.f. Theorem 22.16 in \cite{RFV3} ).
		Then by Lemma \ref{key_1} and (\ref{heat_kernel_estimates_1}), we have for $A\le  \frac{1}{20}T^{-\frac{1}{2}}r_0$ and $A$ is large enough
		\begin{equation}\label{l_1}
			\begin{aligned}
				\left. \int_{M} v h_1 d V_t\right|_{t=0}
				&\ge 	 \mathcal{W}(M, g(0), \tilde{u}^{\frac{1}{2}}(0),1)\left.\int_{M} u h_1 d V_t \right|_{t=0}\\
				&\ \ \ - \left.\int_{\{d_{g(0)}(p,x)\ge 10A\}}\left\{4 |\nabla \sqrt{h_1}|^2-h_1 \log h_1\right\} ud V_t \right|_{t=0}\\
				&\ge 	 - \left.\int_{\{d_{g(0)}(p,x)\ge 10A\}}\left\{4 |\nabla \sqrt{h_1}|^2-h_1 \log h_1\right\} u d V_t \right|_{t=0}\\
				&\ge 	 - \left.\left(\frac{2}{5A^2}+e^{-1} \right)\int_{M\backslash \Omega'_{0}} ud V_t\right|_{t=0}\\
				&\ge 	 - \left(\frac{2}{5A^2}+e^{-1} \right)\int_{\{d_{g(0)}(p,x)\ge 10A\}} Ce^{-\frac{d_{g(0)}(p,x)^2}{D_1}}dV_0\\
				&\ge 	 -  ce^{-cA^2}
			\end{aligned}
		\end{equation}
		where we used  (\ref{heat_kernel_estimates_1}), (\ref{curvature_bound}), the volume comparison, $|\nabla \sqrt{h_1}|_{g(0)}^2\le  \frac{\left(\psi^{\prime}\right)^2}{(10A)^2\psi}\le \frac{1}{10A^2}$ and $-h_1 \log h_1\le e^{-1}$.	
		Setting $$\phi(x,t)=e^{-\frac{t}{10A^2}}\psi\left(\frac{d_{g(t)}(x,p)+2 K \sqrt{t}}{10 A}\right).$$
		We have
		\begin{align*}
			&\left(\frac{\partial}{\partial t}-\Delta\right)\psi\left(\frac{d_{g(t)}(x,p)+2 K\sqrt{t}}{10 A}\right)\\
			=&\frac{1}{10 A}\left(\left(\frac{\partial}{\partial t}-\Delta\right) d_{g(t)}(x,p)+\frac{K}{\sqrt{t}}\right) \psi^{\prime}-\frac{1}{(10 A)^2} \psi^{\prime \prime} \leq \frac{\psi}{10A^2},
		\end{align*}
		where we used $\left(\frac{\partial}{\partial t}-\Delta\right) d_{g(t)}(x,p)+\frac{K}{\sqrt{t}}\ge 0$ (see Lemma 8.3 in \cite{P1} or Section 17 in \cite{H2}). Then we have $\left(\frac{\partial}{\partial t}-\Delta\right) \phi \leq 0$.
		For $v=\left\{(1-t)\left(2 \Delta f-|\nabla f|^2+R\right)+f-n\right\} u$,
		we conclude from  Theorem \ref{Harnack} that
		\begin{equation}\label{qqqqq}
			\begin{aligned}
				\frac{d}{d t} \int_M (-\nu) \phi dV_t =&\int_M\left\{(-\nu)\left(\frac{\partial}{\partial t}-\Delta\right) \phi+\phi \left(-\frac{\partial}{\partial t}-\Delta+R\right)\nu\right\}dV_t \\
				\leq& - 2\int_M (1-t) |Rc+\nabla^2 f-\frac{1}{2(1-t)}g|^2 u\phi dV_t.
			\end{aligned}
		\end{equation}
		Denote $\tau=1-t$.
		We obtain that for $ A\le  \frac{1}{20}T^{-\frac{1}{2}}r_0$ and $A$ large enough
		$$
		2\int^1_0\int_M \tau |Rc+\nabla^2 f-\frac{1}{2\tau}g|^2 u\phi dV_tdt\le \left.\int_M (-\nu) \phi dV_t\right|_{t=0}\le \left.\int_M (-\nu) h_1 dV_t\right|_{t=0}\le ce^{-cA^2}.
		$$
		By (\ref{curvature_bound})
		we obtain that for $A\le  \frac{1}{20}T^{-\frac{1}{2}}r_0$
		and $A$ large enough
		\begin{equation}\label{m_1}
			\int^1_0\int_{\{d_{g(0)}(p,x)\le 8A\}} \tau |Rc+\nabla^2 f-\frac{1}{2\tau}g|^2 udV_tdt\le ce^{-cA^2}.
		\end{equation}
		and
		\begin{equation}\label{m_2}
			\left.\int_{\{d_{g(0)}(p,x)\le 8A\}}  (-\nu)  dV_t\right|_{t=0}\le ce^{-cA^2}.
		\end{equation}
		We take $$h(x)=\psi\left(\frac{d_{g(0)}(p,x)}{4 A}\right)$$ in Lemma \ref{key_2} and get for any $0<t_1<t_2<1$
		\begin{equation}\label{l_2}
			\begin{aligned}
				\left. \int_M \tau R h u d V_t \right|^{t=t_2}_{t=t_1}\geq & - b  \int_{t_2}^{t_1}\int_M \tau\left(|\operatorname{Rc}|^2+(2n+1)|\nabla f|^4+2n^2\left|\nabla^2 f\right|^2\right) h u d V_tdt  \\
				& -3b^{-1} \int_{t_2}^{t_1}\ \int_M \tau\left|\operatorname{Rc}+\nabla^2 f-\frac{1}{2 \tau} g\right|^2 h u d V_t dt\\
				& -   \int_{t_2}^{t_1}\int_M \tau^{-1}\left(\tau\left(-|\nabla f|^2+\triangle f\right)+f-\frac{n}{2}\right) \left(f-\frac{n}{2}\right) h u d V_tdt \\
				& -\int_{t_2}^{t_1}\int_M\tau \left( 2\langle \nabla R,\nabla h \rangle+R\Delta h\right) u d V_t dt+2 \int_{t_2}^{t_1}\int_M \tau\operatorname{Rc}(\nabla f, \nabla h)u d V_tdt \\
				& -\int_{t_2}^{t_1}\int_M\left(-\tau\left(|\nabla f|^2+R\right)+f \right)\langle \nabla f, \nabla h \rangle u d V_tdt,
			\end{aligned}
		\end{equation}
		where $\tau=1-t$ and $	f=-\frac{n}{2} \log \left(4 \pi \tau\right)-\log u$.
		Next we estimate the terms in (\ref{l_2}).
		
		Notice that $u(x,t)\le  C_2(1-t)^{-\frac{n}{2}}$ by (\ref{heat_kernel_estimates_1}) and
		applying the Theorem \ref{gradient_estimates} on time interval $[0,1-\epsilon]$, by letting $\epsilon\to 0$ we conclude that
		\begin{equation}\label{conclude}
			\tau|\nabla f|^2\le c_1(f+1).
		\end{equation}
		It follows from (\ref{heat_kernel_estimates_1}) that
		\begin{equation}\label{m_3}
			\int^{t_2}_{t_1} \int_M \tau|\nabla f|^4 h u d V_t dt
			\le \int^{t_2}_{t_1} \int_M \frac{c_1^2(f+1)^2}{\tau} h u d V_t dt \le c'(1-t_2)^{-1}.
		\end{equation}
		By (\ref{m_1}) and (\ref{heat_kernel_estimates_1}), we have
		\begin{equation}\label{m_4}
			\int^{t_1}_{t_2}\int_M \tau |\nabla^2 f|^2 uhd V_t dt\le \int^{t_1}_{t_2}\int_M \frac{n}{4\tau} uhd V_t dt +K^2+ ce^{-cA^2}  \le \frac{n}{4}(1-t_2)^{-1}+K^2+ ce^{-cA^2} .
		\end{equation}
		It follows from (\ref{m_1}) for $A\le  \frac{1}{20}T^{-\frac{1}{2}}r_0$
		and large enough $A$ that
		\begin{equation}\label{3.16}
			2\int^1_0\int_M \tau |R+\Delta f-\frac{n}{2\tau}|^2 uhd V_t dt\le  ce^{-cA^2},
		\end{equation}
		and we get from (\ref{3.16}) and (\ref{m_2}) that
		\begin{equation*}
			\int^{1}_{0}\int_M \left|\tau\left(-|\nabla f|^2+\triangle f\right)+f-\frac{n}{2}\right|  uh  d V_t \le c e^{-cA^2}
		\end{equation*}
		Since $h=0$ outside $B_{g(0)}(p,8A)$,
		we have for $A\le  \frac{1}{20}T^{-\frac{1}{2}}r_0$
		\begin{equation*}
			\begin{aligned}
				&\quad
				\int^{t_2}_{t_1}\int_M \tau^{-1}\left(\tau\left(-|\nabla f|^2+\triangle f\right)+f-\frac{n}{2}\right) \left(f-\frac{n}{2}\right) h u d V_t \\
				\le &
				\int^{t_2}_{t_1}\int_M \tau^{-1}\left|\tau\left(-|\nabla f|^2+\triangle f\right)+f-\frac{n}{2}\right| \left|\frac{d^2_{g(0)}(p,x)}{D_2\tau}-\frac{n}{2}\right|h u d V_t
				\\
				\le& 	c_2A^2\int^{t_2}_{t_1}\int_M  \tau^{-2}\left(\tau\left(-|\nabla f|^2+\triangle f\right)+f-\frac{n}{2}\right)  h u d V_t
				\\
				\le& c_2A^2 (1-t_2)^{-2}\int^{t_2}_{t_1}\int_M  \left(\tau\left(-|\nabla f|^2+\triangle f\right)+f-\frac{n}{2}\right)  h u d V_t\\
				\le& cA^2 (1-t_2)^{-2} e^{-cA^2}\le c' (1-t_2)^{-2} e^{-c'A^2}.
			\end{aligned}
		\end{equation*}
		for some $c''<c$ and large enough $A$.
		Without loss of generality and for the season of simplicity, we still use $c$ instead of $c''$ and
		\begin{equation}\label{m_5}
			\begin{aligned}
				\int^{t_2}_{t_1}\int_M \tau^{-1}\left(\tau\left(-|\nabla f|^2+\triangle f\right)+f-\frac{n}{2}\right) \left(f-\frac{n}{2}\right) h u d V_t\le c (1-t_2)^{-2} e^{-cA^2}.
			\end{aligned}
		\end{equation}
		Just from (\ref{curvature_bound}), we have $|\nabla h|(g(t))\le c_3$ and $|\Delta h|(g(t))\le c_3$ for $t\in [0,1]$. And notice that
		$|\nabla h|= 0$ and $|\Delta h|= 0$ on if $d_{g(0)}(p,x)\le 4A$ and for $A$ is large enough, we have for any $0<t_1<t_2<1$
		\begin{equation}\label{m_6}
			\begin{aligned}
				&\quad\left|\int^{t_2}_{t_1}\int_M\left(2\langle \nabla R,\nabla h \rangle+R\Delta h\right)  u d V_t dt\right|\\
				\le &
				3Kc_3\int^{t_2}_{t_1}\int_{d_{g(0)}(p,x)\ge 4A} u d V_t dt\\
				\le& 3Kc_3C_2\int^{t_2}_{t_1}\int_{d_{g(0)}(p,x)\ge 4A} \tau^{-\frac{n}{2}} e^{-\frac{d^2_{g(0)}(p,x)}{D_1\tau}}d V_t dt\\
				\le& 3Kc_3C_2\int^{t_2}_{t_1}\int_{d_{g(0)}(p,x)\ge 4A}  e^{-D_1^{-1}d^2_{g(0)}(p,x)}d V_t dt\\
				\le& c e^{-cA^2},
			\end{aligned}
		\end{equation}
		where we used
		$\tau^{-q} e^{-\frac{cx}{\tau}} \le   e^{-cx}$ for any $0<\tau \le 1$ if $x\ge \frac{q}{c}$ . Similarly, we have
		\begin{equation}\label{m_7}
			\begin{aligned}
				&\quad\left|\int^{t_2}_{t_1}\int_M \tau\operatorname{Rc}(\nabla f, \nabla h)u d V_t dt\right|
				\\
				\le &
				c_3K\int^{t_2}_{t_1}\int_{d_{g(0)}(p,x)\ge 4A}\tau |\nabla f|u d V_t dt\\
				\le& c_1^{\frac{1}{2}}c_3 K\int^{t_2}_{t_1}\int_{d_{g(0)}(p,x)\ge 4A} \tau^{\frac{1}{2}} (f+1)^{\frac{1}{2}} ud V_t dt
				\\
				\le& c_1^{\frac{1}{2}}c_3 K\int^{t_2}_{t_1}\int_{d_{g(0)}(p,x)\ge 4A} \tau^{\frac{1}{2}} (\frac{d^2_{g(0)}(p,x)}{D_2\tau}+1)^{\frac{1}{2}} \tau^{-\frac{n}{2}} e^{-\frac{cd^2_{g(0)}(p,x)}{\tau}}d V_t dt\\
				\le& c e^{-cA^2}
			\end{aligned}
		\end{equation}
		and
		\begin{equation}\label{m_8}
			\begin{aligned}
				&\quad \left|\int^{t_2}_{t_1}\int_M\left(-\tau\left(|\nabla f|^2+R\right)+f \right)\langle \nabla f, \nabla h \rangle u d V_t dt\right|\\
				\le &
				\int^{t_2}_{t_1}\int_{d_{g(0)}(p,x)\ge 4A}
				\left(K|\nabla f|+|\nabla f|^3+f|\nabla f|\right) u d V_t dt\\
				\le& c e^{-cA^2}.
			\end{aligned}
		\end{equation}
		
		Combing with (\ref{m_1})-(\ref{m_8}) and take $t_2=1-e^{-\frac{c}{4} A^2}$ and $b=e^{-\frac{c}{2} A^2}$ in (\ref{l_2}),
		we get for any $0<t_1\le t_2$, $A\le  \frac{1}{20}T^{-\frac{1}{2}}r_0$ and large enough $A$
		\begin{equation}\label{scalar_upper_bound}
			\begin{aligned}
				&\left. \int_M (1-t) R h u d V_t \right|_{t=t_1} \\
				\le & 	\left. \int_M (1-t) R h u d V_t \right|_{t=t_2}+c e^{-c A^2}\left((1-t_2)^{-2}+b^{-1}+1\right)+cb\left((1-t_2)^{-1}+1\right)\\
				\le& ce^{-\frac{c}{4} A^2}.
			\end{aligned}
		\end{equation}
		Then for $A\le  \frac{1}{20}T^{-\frac{1}{2}}r_0$ and large enough $A$
		\begin{equation}\label{R_upper_bound}
			\begin{aligned}
				\left. \int_M  R h u d V_t \right|_{t=\frac{1}{2}} \le 2ce^{-\frac{c}{4} A^2}.
			\end{aligned}
		\end{equation}
		We calculate that
		\begin{equation}\label{R_estimates}
			\begin{aligned}
				\frac{d}{dt}\int_M  R h u d V_t &=\int_M u\left(\frac{\partial}{\partial t}-\Delta\right)( R h) d V_t\\
				&= \int_M 2|Rc|^2 u h d V_t-\int_M\left( 2\langle \nabla R,\nabla h \rangle+R\Delta h\right) u d V_t .
			\end{aligned}
		\end{equation}
		It follows that $A\le  \frac{1}{20}T^{-\frac{1}{2}}r_0$ and large enough $A$
		\begin{equation}\label{bbq}
			\begin{aligned}
				&\quad\int^{\frac{1}{2}}_{\frac{1}{4}}	\int_M 2|Rc|^2 u h d V_tdt\le \int^{\frac{1}{2}}_0	\int_M 2|Rc|^2 u h d V_tdt\\	
				&=\left. \int_M  R h u d V_t\right|_{t=\frac{1}{2}}-\left. \int_M  R h u d V_t\right|_{t=0} +\int^{\frac{1}{2}}_0\int_M\left( 2\langle \nabla R,\nabla h \rangle+R\Delta h\right) u d V_t dt\\
				&\le 3ce^{-\frac{c}{4} A^2},
			\end{aligned}
		\end{equation}
		where we used (\ref{m_6}) and $R= 0$ at $t=0$ on $B_{h_0}(p,r_0)$. Then there exists
		$\xi\in [\frac{1}{4},\frac{1}{2}]$ such that
		\begin{equation}\label{l_4}
			\left.	\int_M 2|Rc|^2 u h d V_t\right|_{t=\xi}\le  12ce^{-\frac{c}{4} A^2}.
		\end{equation}
		
		Let $ \bar{u}(p, s; x,t)$ be the heat kernel of the conjugate heat equation (\ref{con_heat_kernel}) centered at $(p,s)$ with respect to $\bar{g}(t)$.
		Noted that under the rescaling $g(t)=T^{-1}\bar{g}(Tt)$,
		we have $T^{\frac{n}{2}}\bar{u}(p, T; x,\xi T) =u(p,1; x,\xi )$. Then it follows from (\ref{l_4})
		for $A\le  \frac{1}{20}T^{-\frac{1}{2}}r_0$ and $\xi\in [\frac{1}{4},\frac{1}{2}]$ that
		$$	2\int_M  |\overline{Rc}|^2(x,\xi T) \bar{u}(p, T; x,\xi T) h(x) d \overline{V}_{T\xi}(x)\le  12c T^{-2}e^{-\frac{c}{4}A^2}.$$
		Now take $A=\frac{1}{20}r_0T^{-\frac{1}{2}}$, we have
		$$	2\int_{M}  |\overline{Rc}|^2(x,\xi T) \bar{u}(p, T; x,\xi T) h(x) d \overline{V}_{T\xi}(x) \le  12c T^{-2}e^{-\frac{cr_0^2}{1600T}}$$
		with $\xi\in [\frac{1}{4},\frac{1}{2}]$, 
		where $\overline{Rc}$ denotes the Ricci curvature with respect to
		$\bar{g}(t)$.
		Take $T\to 0$,
		we have
		\begin{equation}\label{key}
			|\overline{Rc}|^2(p,0)=\lim\limits_{T\to 0}\int_{M}  |\overline{Rc}|^2(x,\xi T) \bar{u}(p, T; x,\xi T) h(x) d \overline{V}_{T\xi}(x) =0.
		\end{equation}
		The reason why (\ref{key}) holding is the following: Notice that $h(x)\equiv 1$ on
		$B_{h}(p,\frac{1}{5}r_0)$. For any $\epsilon>0$, we can choose $\delta>0$
		such that if $d_{h}(p,x)<\delta<\frac{1}{5}r_0$ and $T<\delta$, such that $\left||\overline{Rc}|^2(x,\xi T)- |\overline{Rc}|^2(p,0)\right|<\epsilon$ and $h(x)=1 $.
		Then
		\begin{equation*}
			\begin{aligned}
				&\quad\left|\int_{M}  |\overline{Rc}|^2(x,\xi T) \bar{u}(p, T; x,\xi T) h(x) d \overline{V}_{T\xi}(x) -|\overline{Rc}|^2(p,0)\right|\\	
				=	&\left|\int_{M}\left(  |\overline{Rc}|^2(x,\xi T) h(x)-|\overline{Rc}|^2(p,0)\right) \bar{u}(p, T; x,\xi T) d \overline{V}_{T\xi}(x) \right|\\
				\le 	&\int_{B_{h}(p,\delta)}\left| |\overline{Rc}|^2(x,\xi T) -|\overline{Rc}|^2(p,0)\right| \bar{u}(p, T; x,\xi T) d \overline{V}_{T\xi}(x) \\
				&+ \int_{M\backslash B_{h}(p,\delta)}\left| |\overline{Rc}|^2(x,\xi T)h(x) -|\overline{Rc}|^2(p,0)\right| \bar{u}(p, T; x,\xi T) d \overline{V}_{T\xi}(x) \\
				\le	& \epsilon+2K^2\int_{M\backslash B_{h}(p,\delta)} C_2 (T-\xi T)^{-\frac{n}{2}} e^{-\frac{d_{\bar{g}(0)}^2(p, x)}{DT(1-\xi)}} d \overline{V}_{T\xi}(x)	\le  2\epsilon,
			\end{aligned}
		\end{equation*}
		as $T\to 0$.
		Since $p$ is chosen arbitrarily, we conclude that $|\overline{Rc}|\equiv 0$ on $V$. Then the theorem holds by the Bishop-Gromov's volume comparison.
	\end{proof}
	
	\section{Proof of Theorem \ref{rigidity_mu} and Theorem \ref{rigidity_log_sobolev}}	
	In this section we give the proofs of Theorem \ref{rigidity_mu} and Theorem \ref{rigidity_log_sobolev}. 
	In order to prove Theorem \ref{rigidity_mu},  firstly we need to prove some 
	lemmas.
	
	 In \cite{BCL}
	Bakry, Concordet and Ledoux  proved  that if an $n$-dimensional complete Riemannian manifold $(M,g)$ with non-negative Ricci curvature satisfies the  $L^2$-logarithmic Sobolev inequality with the best constant of $\mathbb{R}^n$, then $(M,g)$ must be isometric to $\mathbb{R}^n$. Their proof relies on a property of the large time behavior of heat kernel, which proved by Li \cite{LiPeter} that for a complete manifold with non-negative Ricci curvature 
	\begin{equation}\label{large_time_behavior}
		\lim _{t \rightarrow \infty} Vol(B(p, \sqrt{t}) \widetilde{H}(x,t;p, 0)=\frac{\left|\omega_n\right|}{(4 \pi)^{n / 2}},
	\end{equation}
	where  $\widetilde{H}(x,t;p,0)$ is  the heat kernel $\left(\frac{\partial}{\partial t}-\Delta\right)\widetilde{H}=0$ on $M$. Moreover, the  $L^2$-logarithmic Sobolev inequality with the best constant of $\mathbb{R}^n$ implies
	that
	\begin{equation}\label{l_uppp}
		\widetilde{H}(x,t;p, 0)\le\frac{1}{(4 \pi t)^{n / 2}} 
	\end{equation}
	(cf. Proposition 1.4 in \cite{BCL} ).
	It follows from  (\ref{large_time_behavior}) and (\ref{l_uppp}) that
	\begin{equation}\label{asymp}
		\lim _{r \rightarrow \infty}\frac{Vol(B(p, r))}{r^n}\ge \omega_n
	\end{equation}
	and hence $(M,g)$ must be isometric to $\mathbb{R}^n$ by the Bishop-Gromov 
	volume comparison theorem. One may see \cite{Ni2} and \cite{BKT} for the other proofs. 
	However, their methods could not be used for the  local case. We 
	extend their result to local case by using the small time behavior of the Dirichlet heat kernel. This lemma will be needed in the proof of  Theorem \ref{rigidity_mu},

	\begin{lem}\label{Ricci_flat_flat}
		Let $M$ be an $n$-dimensional manifold. Suppose that for some open subset $V\subset M$   satisfying
		\begin{equation}\label{log_sobolev_with_best}
			\int_V \varphi^2 \log \varphi^2 d vol\le \int_{V}4\tau |\nabla \varphi|^2 d vol-n-\frac{n}{2} \log (4 \pi \tau),
		\end{equation}
		for all $\varphi\in\mathcal{S}(V)$ and $\tau<\tau_0$ and
		for some $\tau_0>0$. Then $R(x)\le 0$ for all $x\in V$, where $R$ is the scalar curvature. Moreover, if we assume additionally that
		\begin{equation*}
			Rc(x)\ge 0
		\end{equation*}	
		for all $x\in V$,
		then $V$ must be flat.	
	\end{lem}
	\begin{proof}
		We take an arbitrary fixed point $p \in V $ and relatively compactly
		small geodesic ball $B(p,2r_0)\subset V$.	
		Let $H(x,t)\doteq H(x,t;p,0)$ be  the Dirichlet heat kernel centered at $(p,0)$  of
		$$
		\left\{\begin{array}{l}
			\left(\frac{\partial}{\partial t}-\Delta\right) H=0 \text{ in } B(p,r_0),\\
			H=0 \text{ on } \partial B(p,r_0).
		\end{array}\right.
		$$	
		Firstly, we follow the original ideas of of Davies \cite{Davies} (also see Theorem 1.2 in \cite{BCL}) to show  (\ref{log_sobolev_with_best}) implies 
		that 
		\begin{equation}\label{heat_kernel_up_best}
			H(x, T) \leq \frac{1}{(4 \pi T)^{n / 2}} 
		\end{equation}
		for all $x\in V$ and $T<\tau_0$.	
		Given $T \in(0,1]$ and $t \in(0, T)$, we take $p(t)=$ $T /(T-t)$ so that $p(0)=1$ and $p(T)=\infty$. By direct computation, we obtain that
		$$
		\begin{aligned}
			\partial_t\|H\|_{p(t)}= & \partial_t\left(\int_{V} H^{p(t)}(x, t) dV\right)^{\frac{1}{p(t)}} \\
			=- & \frac{p^{\prime}(t)}{p^2(t)}\|H\|_{p(t)} \log \int_{V} H^{p(t)}(x, t) dV+\frac{1}{p(t)}\left(\int_{V} H^{p(t)}(x, t) dV\right)^{\frac{1}{p(t)}-1} \\
			& \times\left[\int_{V} H^{p(t)}(\log H) p^{\prime}(t) d V+p(t) \int_{V}H^{p(t)-1}\Delta H dV\right] .
		\end{aligned}
		$$
		Using integration by
		parts and $p(t)>1$ for $t>0$, we have for $t>0$ 		
		\begin{equation}\label{ibp}
			\begin{aligned}
				& p^2(t)\|H\|_{p(t)}^{p(t)} \partial_t \log \|H\|_{p(t)} \\
				& =-p^{\prime}(t)\|H\|_{p(t)}^{p(t)} \log \int_{V} H^{p(t)} dV+p(t) p^{\prime}(t) \int_{V}H^{p(t)} \log H dV \\
				& \quad-4(p(t)-1) \int_{V}\left|\nabla\left(H^{\frac{p(t)}{2}}\right)\right|^2 dV.
			\end{aligned}
		\end{equation}	
		Merging the first two terms on the righthand side of the above equality and making the substitution $v=H^{\frac{p(t)}{2}} /\left\|H^{\frac{p(t)}{2}}\right\|_2$, we arrive at for $t>0$, after dividing by $\|H\|_{p(t)}^{p(t)}$,
		$$
		\begin{aligned}
			\partial_t \log \|H\|_{p(t)} =\frac{p^{\prime}(t)}{	p^2(t)}\left( \int_{V} v^2 \log v^2 dV-\frac{4(p(t)-1)}{p'(t)} \int_{V}|\nabla v|^2 dV\right).
		\end{aligned}
		$$
		Note that for $0<t<T$
		$$
		\begin{gathered}
			\frac{4(p(t)-1)}{p^{\prime}(t)}=\frac{4 t(T-t)}{T} \leq T < \tau_0,
		\end{gathered}
		$$
		and  $\frac{p^{\prime}(t)}{ p^2(t)} =\frac{1}{T}$. Applying (\ref{log_sobolev_with_best}), we get for  $0<t\le T<\tau_0$ 
		$$
		\partial_t \log \|H\|_{p(t)} \leq \frac{1}{T}\left(-\frac{n}{2} \log \frac{4\pi  t(T-t)}{T} -n\right) .
		$$
		This yields (\ref{heat_kernel_up_best}), after integration from $t=\epsilon$ to $t=T$ for $T<\tau_0$  and let $\epsilon\to 0$.
		
		We extend $B(p,r_0)$ to a smooth compact manifold $\widetilde{\mathcal{M}}$ without boundary. Indeed, we may do this by extending $B(p,r_0)$ to a collar past its boundary, doubling the extension, and then extending the metric so that it is a product in a collar of the new boundary and then doubling the metric. Denote
		$
		\widetilde{H}
		$
		be the heat kernel for $\left(\frac{\partial}{\partial t}-\Delta\right)\widetilde{H}=0$ in $\widetilde{\mathcal{M}}$.
		Let
		$
		f_{p, 0}: B(p,r_0)\times[0, T] \rightarrow \mathbb{R}
		$
		be the solution to $\left(\frac{\partial}{\partial t}-\Delta\right) f_{p, 0}=0$ with the boundary conditions
		$$
		\begin{aligned}
			& f_{p, 0}(x, 0)=0 \quad \text { for } x \in B(p,r_0), \\
			& f_{p, 0}(x, t)=-\widetilde{H}(x, t ; p, 0) \quad \text { for } x \in \partial B(p,r_0) \text { and } t \in(0, T] .
		\end{aligned}
		$$
		Then we have
		$$
		H(x, t ; p,0) = \widetilde{H}(x, t ; p,0)+f_{p, 0}(x, t) 
		$$
		(cf. Lemma 24.24 and (24.16) in \cite{RFV3}).
	By the expansion of the heat kernel of closed manifolds(cf. \cite{berger} P215 or \cite{RFV3} P253) 
		\begin{equation}\label{heat_kernel_expansion}
			\widetilde{H}(p,t;p,0)=\frac{1}{(4 \pi t)^{n / 2}}\left(\sum_{j=0}^k t^j \phi_j(p, p)+O\left(t^{k+1}\right)\right)	
		\end{equation}
		with  $\phi_0(p, p)=1$ and $\phi_1(p, p)=\frac{1}{6} R(p)$ and $\phi_2(p,p)=\left(\frac{1}{30} \Delta R+\frac{1}{72} R^2-\frac{1}{180}|\mathrm{Rc}|^2+\frac{1}{180}|\mathrm{Rm}|^2\right)(p)$.  Note that $f_{p, 0}(p, t)\to 0$ as $t\to 0$. 
		Then we
		can conclude from (\ref{heat_kernel_expansion}) that
		$$
		H(p,t;p,0)>\frac{1}{(4 \pi t)^{n / 2}}
		$$
		for $t$  sufficient small if  $R(p)> 0$.  This contradicts to (\ref{heat_kernel_up_best}) and we get  $R(x)\le 0$ for all $x\in V$ since $p$ is arbitrarily choosen on $V$. Moreover, if
		$
		Rc(x)\ge 0
		$
		for all $x\in V$, then $Rc \equiv 0$  on $V$.	
		If $|Rm|(p)\neq 0$, then $\phi_1(p, p)=0$ and $\phi_2(p,p)=\frac{1}{180}|\mathrm{Rm}|^2(p)>0$ since $Rc\equiv 0$ on $V$. Then we
		can conclude from (\ref{heat_kernel_expansion}) that
		$$
		H(p,t;p,0)>\frac{1}{(4 \pi t)^{n / 2}}
		$$
		for $t$  sufficient small if  $|Rm|(p)\neq 0$ and $n\ge 4$. This contradicts to (\ref{heat_kernel_up_best}) and we get $|Rm|\equiv 0$ on $V$ when  $n\ge 4$. If $n\le 3$, the lemma holds obviously. 
	\end{proof}

	We also need  get a Laplacian upper bound of the positive solutions to the conjugate heat equations under the Ricci flow, for which the estimate was proved by Hamilton \cite{H3} for the case of  heat equation closed manifolds.
	\begin{thm}	\label{laplace_estimates}
		Let $g(t)$, $t\in [0,T]$, be the complete solution to the Ricci flow on an $n$-dimensional manifold $M$ with bounded sectional curvature.
		Assume that
		$$
		\max\{|Rc(g(t))|,|\nabla R (g(t))|, |\Delta R (g(t))|\}\le K
		$$
		on $M\times [0,T]$.
		Suppose $u>0$ is a solution of
		$$
		\square^* u=\left(-\partial_t-\Delta+R\right) u=0,
		$$
		and $u \leq A$  on $M\times [0,T]$. Then
		$$
		(T-t)\left(\frac{\Delta u}{u}+\frac{|\nabla u|^2}{u^2}\right) \leq N+Q \log \left(\frac{A}{u}\right) ,
		$$
		where $Q$ and $N$ are the positive constants only depending on $n$, $K$ and $T$.
	\end{thm}
	\begin{proof}
		Let $\tau=T-t.$
		Let $L \doteqdot \frac{\partial}{\partial \tau}-\Delta-2 \nabla \log u \cdot \nabla$.
		Since $\frac{\partial u}{\partial \tau}=\Delta u-Ru$, we have $\frac{\partial}{\partial \tau} \log u=\Delta \log u+|\nabla \log u|^2-R$. Using the Bochner formula, we compute
		$$
		\begin{aligned}
			\frac{\partial}{\partial \tau}|\nabla \log u|^2= & 2 \nabla \log u \cdot \nabla\left(\Delta \log u+|\nabla \log u|^2-R\right)-2 \operatorname{Rc}(\nabla \log u, \nabla \log u) \\
			= & \Delta|\nabla \log u|^2+2 \nabla \log u \cdot \nabla|\nabla \log u|^2 -2\nabla \log u\cdot \nabla R\\
			& -2|\nabla \nabla \log u|^2-4 \operatorname{Rc}(\nabla \log u, \nabla \log u)
		\end{aligned}
		$$
		which yields
		\begin{equation}\label{q_1}
			L\left(|\nabla \log u|^2\right)=-2|\nabla \nabla \log u|^2 -2\nabla \log u\cdot \nabla R-4 \operatorname{Rc}(\nabla \log u, \nabla \log u) .
		\end{equation}
		We also compute
		$$
		\begin{aligned}
			\frac{\partial}{\partial \tau} \log \left(\frac{A}{u}\right)   =\Delta \log \left(\frac{A}{u}\right)+2 \nabla \log u \cdot \nabla \log \left(\frac{A}{u}\right)+|\nabla \log u|^2+R,
		\end{aligned}
		$$
		which implies
		\begin{equation}\label{q_5}
			L\left(\log \left(\frac{A}{u}\right)\right)=|\nabla \log u|^2+R.	
		\end{equation}
		We also have
		\begin{equation}\label{q_2}
			\begin{aligned}
				\frac{\partial}{\partial \tau}(\Delta \log u)= & \Delta\left(\Delta \log u+|\nabla \log u|^2-R\right) -2R_{ij}\nabla_i\nabla_j \log u\\\
				= & \Delta(\Delta \log u)+2 \nabla \log u \cdot \nabla(\Delta \log u)+2|\nabla \nabla \log u|^2\\
				& +2 \operatorname{Rc}(\nabla \log u, \nabla \log u)  -\Delta R-2R_{ij}\nabla_i\nabla_j \log u.
			\end{aligned}
		\end{equation}
		Hence, combining this with (\ref{q_1}), we have for all $0<a<2$
		$$
		\begin{aligned}
			&\quad	L\left(\Delta \log u+2|\nabla \log u|^2\right) \\
			& =-2|\nabla \nabla \log u|^2-6 \operatorname{Rc}(\nabla \log u, \nabla \log u)-4\nabla \log u\cdot \nabla R  -\Delta R-2R_{ij}\nabla_i\nabla_j \log u\\
			& \leq-\frac{2-a}{n}(\Delta \log u)^2+8K|\nabla \log u|^2+3K+\frac{K^2}{a} ,
		\end{aligned}
		$$
		where we used $4\nabla \log u\cdot \nabla R\le 2K+2K|\nabla \log u|^2$
		and $-2R_{ij}\nabla_i\nabla_j \log u\le a|\nabla\nabla \log u|^2+\frac{K^2}{a}$.
		Now let
		$$
		F \doteqdot \tau\left(\Delta \log u+2|\nabla \log u|^2\right)-N-Q \log \left(\frac{A}{u}\right) ,
		$$
		where $Q=8KT+4$ and $N$ is a positive constant which  will be determined later.
		we shall show that $L F \leq 0$ wherever $F \geq 0$.
		We have
		$$
		L F \leq-\frac{(2-a) \tau}{n}(\Delta \log u)^2+\Delta \log u-2|\nabla \log u|^2 +(3K+\frac{K^2}{a})\tau+QK.
		$$
		If $F \geq 0$, then
		$$
		\begin{gathered}
			-2|\nabla \log u|^2 \leq \Delta \log u-\frac{N}{\tau} .
		\end{gathered}
		$$
		This implies when  $F \geq 0$
		$$
		\begin{aligned}
			L F &\leq -\frac{(2-a) \tau}{n}(\Delta \log u)^2+2\Delta \log u +(3K+\frac{K^2}{a})\tau+QK-\frac{N}{\tau} \\
			&= -\frac{(2-a) \tau}{n}\left(\Delta \log u-\frac{n}{(2-a)\tau}\right)^2+\frac{n}{(2-a)\tau}+(3K+\frac{K^2}{a})\tau+QK-\frac{N}{\tau}<0 ,
		\end{aligned}
		$$
		if $N>\frac{n}{(2-a)}+(3K+\frac{K^2}{a})T^2+QKT$. Then Theorem \ref{laplace_estimates} follows from the maximum principle.
		
	\end{proof}
	We also need an opposite estimate of Lemma \ref{key_2}.
	\begin{lem}\label{key_11}
		Under the assumptions of Lemma \ref{key_2},			
		we have for any $b>0$ the following estimate holds
		\begin{equation}\label{key_3_main}
			\begin{aligned}
				\frac{d}{d t} \int_M \tau R h u d V_t \le &  b  \int_M \tau\left(|\operatorname{Rc}|^2+(2n+1)|\nabla f|^4+2n^2\left|\nabla^2 f\right|^2\right) h u d V_t  \\
				& +3b^{-1}  \int_M \tau\left|\operatorname{Rc}+\nabla^2 f-\frac{1}{2 \tau} g\right|^2 h u d V_t \\
				& -  \int_M\left(\tau\left(-|\nabla f|^2+\triangle f\right)+f-\frac{n}{2}\right)\left(|\nabla f|^2-\Delta f\right) h u d V_t \\
				& -\tau \int_M\left( 2\langle \nabla R,\nabla h \rangle+R\Delta h\right) u d V_t +2 \tau\int_M \operatorname{Rc}(\nabla f, \nabla h)u d V_t \\
				& -\int_M\left(-\tau\left(|\nabla f|^2+R\right)+f \right)\langle \nabla f, \nabla h \rangle u d V_t,
			\end{aligned}
		\end{equation}
		where $\tau=T-t$ and $	f=-\frac{n}{2} \log \left(4 \pi \tau\right)-\log u$.
	\end{lem}
	\begin{proof}
		We get from (\ref{p_1}) that
		for any $b>0$,
		\begin{equation}\label{est_101}
			\begin{aligned}
				&\ \ \ \	\tau  \frac{d}{d t} \int_M \tau R h u d V_t \\
				& \le b \tau^2 \int_M|\operatorname{Rc}|^2 h u d V_t+b^{-1} \tau^2 \int_M\left|\operatorname{Rc}+\nabla^2 f-\frac{1}{2 \tau} g\right|^2 h u d V_t\\
				&\ \ -2 \tau^2 \int_M \operatorname{Rc} \cdot \nabla^2 f h u d V_t-\tau^2 \int_M\left( 2\langle \nabla R,\nabla h \rangle+R\Delta h\right) u d V_t.
			\end{aligned}
		\end{equation}	
		Similar as (\ref{est_2}), we have
		\begin{equation}\label{est_102}
			\begin{aligned}
				&\ \ \ \	2 \tau^2 \int_M \operatorname{Rc} \cdot \nabla^2 f h u d V_t\\
				\ge & -b \tau^2 \int_M|\nabla f|^4 h u d V_t-b^{-1} \tau^2 \int_M\left|\operatorname{Rc}+\nabla^2 f-\frac{1}{2 \tau} g\right|^2 h u d V_t \\
				& +\tau \int_M \nabla\left(-\tau\left(|\nabla f|^2+R\right)+f \right) \cdot \nabla f h u d V_t
				-2 \tau^2 \int_M \operatorname{Rc}(\nabla f, \nabla h)u d V_t ,
			\end{aligned}
		\end{equation}
		We modify the calculations in  (\ref{est_3}) as
		\begin{equation}\label{est_103}
			\begin{aligned}
				&\ \ \ \  \tau \int_M \nabla(-\left.\tau\left(|\nabla f|^2+R\right)+f \right) \cdot \nabla f h u d V_t\\
				&=-\tau^2 \int_M\left(\operatorname{Rc}+\nabla^2 f-\frac{1}{2 \tau} g\right) \cdot\left(\left(|\nabla f|^2-\triangle f\right) g\right) h u d V_t \\
				&\ \ \ \ +\tau \int_M\left(\tau\left(-|\nabla f|^2+\triangle f\right)+f-\frac{n}{2}\right)\left(|\nabla f|^2-\Delta f\right) h u d V_t\\
				&\ \ \ \ +\tau \int_M\left(-\tau\left(|\nabla f|^2+R\right)+f \right)\langle \nabla f, \nabla h \rangle u d V_t \\
				& \ge -2 n \tau^2 b \int_M|\nabla f|^4 h u d V_t-2 n^2 \tau^2 b \int_M\left|\nabla^2 f\right|^2 h u d V_t-\tau^2 b^{-1} \int_M\left|\operatorname{Rc}+\nabla^2 f-\frac{1}{2 \tau} g\right|^2 h u d V_t \\
				& \quad  +\tau \int_M\left(\tau\left(-|\nabla f|^2+\triangle f\right)+f-\frac{n}{2}\right)\left(|\nabla f|^2-\Delta f\right) h u d V_t \\
				& \quad 	+\tau \int_M\left(-\tau\left(|\nabla f|^2+R\right)
				+f \right)\langle \nabla f, \nabla h \rangle u d V_t.
			\end{aligned}
		\end{equation}
		Then (\ref{key_3_main})	holds by
		combining (\ref{est_101}), (\ref{est_102}) and (\ref{est_103}).	
	\end{proof}

	Now we give the proof of Theorem \ref{rigidity_mu}.
	
	\begin{proof}[Proof of Theorem \ref{rigidity_mu}]
		We use the same notations as the proof Theorem  \ref{rigidity_2}.
		The proof of Theorem \ref{rigidity_mu} is similar to Theorem  \ref{rigidity_2} and we only need modify little as the follows.
		Firstly, (\ref{wge0}) still holds since for $T\le \tau_0$
		\begin{equation}\label{wge02}
			\begin{aligned}
				\mathcal{W}(M, g(0), \tilde{u}^{\frac{1}{2}}(0),1)
				&= \mathcal{W}(B_{h_0}(p,r_0), \bar{g}(0), \tilde{u}^{\frac{1}{2}}(0),T)\\
				&\ge \boldsymbol{\nu}(V, h_0, \tau_0)\\
				&\ge 0.
			\end{aligned}
		\end{equation}
		And the estimates (\ref{m_1})-(\ref{m_8}) also hold by the same calculates as in the the proof Theorem  \ref{rigidity_2}. Note that we need $R\ge 0$ at $t=0$ to get (\ref{bbq}) in the proof of Theorem \ref{rigidity}. Instead of (\ref{bbq}), we use Lemma \ref{key_3_main} to do an opposite estimate for (\ref{scalar_upper_bound}).
		By Lemma \ref{laplace_estimates}, Theorem \ref{gradient_estimates} and (\ref{heat_kernel_estimates_1}), we have
		\begin{equation}\label{last_111}
			\Delta f\ge -\frac{c_4(f+1)}{\tau}.
		\end{equation}
		And we get from the Theorem \ref{Harnack}, Theorem \ref{gradient_estimates} and (\ref{heat_kernel_estimates_1}) that
		\begin{equation}\label{last_222}
			\Delta f\le \frac{c_5(f+1)}{\tau}.
		\end{equation}
		
		Since $h=0$ outside $B_{g(0)}(p,8A)$ and by  (\ref{heat_kernel_estimates_1}), (\ref{conclude}), (\ref{last_111}) and (\ref{last_222}), we
		have for $A\le  \frac{1}{20}T^{-\frac{1}{2}}r_0$, $A$ large enough and  any $0<t_1<t_2<1$
		\begin{equation*}
			\begin{aligned}
				&\quad
				\int^{t_2}_{t_1}\int_M\left(\tau\left(-|\nabla f|^2+\triangle f\right)+f-\frac{n}{2}\right)\left(|\nabla f|^2-\Delta f\right) h u d V_t  dt\\
				\le &
				\int^{t_2}_{t_1}\int_M \left|\tau\left(-|\nabla f|^2+\triangle f\right)+f-\frac{n}{2}\right| \frac{c(f+1)}{\tau}h u d V_t
				\\
				\le& 	c_2A^2\int^{t_2}_{t_1}\int_M  \tau^{-2}\left(\tau\left(-|\nabla f|^2+\triangle f\right)+f-\frac{n}{2}\right)  h u d V_t
				\\
				\le& c_2A^2 (1-t_2)^{-2}\int^{t_2}_{t_1}\int_M  \left(\tau\left(-|\nabla f|^2+\triangle f\right)+f-\frac{n}{2}\right)  h u d V_t\\
				\le& cA^2 (1-t_2)^{-2} e^{-cA^2}\le c' (1-t_2)^{-2} e^{-c'A^2}.
			\end{aligned}
		\end{equation*}
		for some $c''<c$ and large enough $A$. For the season of simplicity, we still use $c$ instead of $c''$ and
		\begin{equation}\label{p_5}
			\begin{aligned}
				\int^{t_2}_{t_1}	\int_M\left(\tau\left(-|\nabla f|^2+\triangle f\right)+f-\frac{n}{2}\right)\left(|\nabla f|^2-\Delta f\right) h u d V_t dt\le c (1-t_2)^{-2} e^{-cA^2}.
			\end{aligned}
		\end{equation}
		
		Combing with (\ref{m_1})-(\ref{m_8}), (\ref{p_5}), Lemma \ref{key_11} and take $t_2=1-e^{-\frac{c}{4} A^2}$ and $b=e^{-\frac{c}{2} A^2}$ in (\ref{key_3_main}),
		we get for any $0<t_1\le t_2$, $A\le  \frac{1}{20}T^{-\frac{1}{2}}r_0$ and large enough $A$
		\begin{equation*}
			\begin{aligned}
				&\left. \int_M (1-t) R h u d V_t \right|_{t=t_1}\\
				 \ge & 	\left. \int_M (1-t) R h u d V_t \right|_{t=t_2}-c e^{-c A^2}\left((1-t_2)^{-2}+b^{-1}+1\right)+cb\left((1-t_2)^{-1}+1\right)\\
				\ge &-ce^{-\frac{c}{4} A^2}.
			\end{aligned}
		\end{equation*}
		Then for $A\le  \frac{1}{20}T^{-\frac{1}{2}}r_0$ and large enough $A$, we  get that
		\begin{equation}\label{R_lower_bound}
			\begin{aligned}
				\left. \int_M  R h u d V_t \right|_{t=0} \ge -2ce^{-\frac{c}{4} A^2}.
			\end{aligned}
		\end{equation}	
		Combining with (\ref{R_lower_bound}) and (\ref{R_upper_bound}) and by the same estimates as (\ref{R_estimates}) and (\ref{bbq}), we still can conclude 
		that
		\begin{equation}
			\begin{aligned}
				\int^{\frac{1}{2}}_{\frac{1}{4}}	\int_M 2|Rc|^2 u h d V_tdt\le 3ce^{-\frac{c}{4} A^2},
			\end{aligned}
		\end{equation}
		Hence the rest proof is the same as the proof of Theorem \ref{rigidity} and we still have that $V$ is Ricci flat. Since $Rc\equiv 0$ on $V$, 
		the condition (\ref{nu_assumption}) becomes the following:
		$$
		\int_V \varphi^2 \log \varphi^2 d vol \le \int_{V}4\tau |\nabla \varphi|^2 d vol-n-\frac{n}{2} \log (4 \pi \tau),
		$$
		for all $\varphi\in\mathcal{S}(V)$ and $\tau<\tau_0$. Then Theorem \ref{rigidity_mu} follows from Lemma \ref{Ricci_flat_flat}.
	\end{proof}

	Finally, we give the proof of  Theorem \ref{rigidity_log_sobolev}.
	\begin{proof}[Proof of Theorem \ref{rigidity_log_sobolev}]
	We get from (\ref{positive_scalar2}) and Lemma \ref{Ricci_flat_flat}
		that  $R\equiv 0$ on $V$.
		 Combining (\ref{log_sobolev}) and $R\equiv 0$ on $V$, we conclude that (\ref{nu_assumption}) holds. Then Theorem \ref{rigidity_log_sobolev}	follows from Theorem \ref{rigidity_mu} directly.
	\end{proof}
	
	\section{the  rigidity theorems for the cases $R\ge -n(n-1)$ and $R\ge n(n-1)$}
	
	In this section  we give the proofs of Theorem \ref{rigidity_Iso_Hn} and Theorem \ref{rigidity_Sn}.
	
	\begin{proof}[Proof of Theorem \ref{rigidity_Iso_Hn}]
		Firstly, under the conformal change $h=\left(\cosh{\frac{d_g(p,x)}{2}}\right)^{-4}g$, we have 
		\begin{equation}\label{lowerbound_R_h}
			\begin{aligned}
				&R_h(x)\\
				=&\left(\cosh{\frac{d_g(p,x)}{2}}\right)^{4}\left\{R_g(x)+(n-1)+2(n-1)\Delta_g d_g(p,x)\tanh{\frac{d_g(p,x)}{2}}-(n-1)^2\left(\tanh{\frac{d_g(p,x)}{2}}\right)^{2}\right\}\\
				\ge& \left(\cosh{\frac{d_g(p,x)}{2}}\right)^{4}\left\{-(n-1)^2+2(n-1)\Delta_g d_g(p,x)\tanh{\frac{d_g(p,x)}{2}}-(n-1)^2\left(\tanh{\frac{d_g(p,x)}{2}}\right)^{2}\right\}
			\end{aligned}	
		\end{equation}
		when $d_g(p,x)$ is smooth at $x\in  B_g(p,r_0)$
		(cf. P192 in \cite{STbook}).
		Note that (\ref{iso_Hn}) implies that
		\begin{equation}\label{volume_lower_bound}
			\frac{\operatorname{Vol_h}(B_h(x,r))}{r^n}\ge \omega_n
		\end{equation} 
		for all $x\in  B_g(p,r_0)$ and $r$ sufficient small.
		It follows from the expansion for volumes of geodesic balls
		$$\operatorname{Vol_h}(B_h(x, r))=\omega_n r^n\left(1-\frac{R_h(x)}{6(n+2)} r^2+O\left(r^3\right)\right)$$
		that $R_h(x)\le 0$ for all  $d_g(p,x)$ is smooth at $x\in B_g(p,r_0)$. Then we can conclude from (\ref{lowerbound_R_h}) that for a.e $x\in B_g(p,r_0)$
		\begin{equation*}
			\Delta_g d_g(p,x)\le (n-1)\coth d_g(p,x).
		\end{equation*}
		Then for $r\le r_0$
		$$
		\frac{d}{d r} Area_g(\partial B_g(p,r))=\int_{\partial B_g(p,r)}\Delta_g d_g\le (n-1)Area_g(\partial B_g(p,r))\coth r, 
		$$
		and 
		it follows that
		$$
		Area_g(\partial B_g(p,r))\le n\omega_n \left(\sinh r\right)^{n-1}.
		$$
		So we have for $r_h=2\tanh \frac{r}{2}$ and $r\le r_0$
		$$
		\begin{aligned}
			Area_h(\partial B_h(p,r_h))=&\left(\cosh{\frac{r}{2}}\right)^{-2(n-1)}Area_g(\partial B_g(p,r))\\
			\le&  n\omega_n\left(\cosh{\frac{r}{2}}\right)^{-2(n-1)} \left(\sinh r\right)^{n-1}\\
			=& n\omega_n  \left(2\tanh{\frac{r}{2}}\right)^{n-1}\\
			=& n\omega_n  r_h^{n-1},
		\end{aligned}
		$$
		and combing this with (\ref{iso_Hn})
		we have $\frac{\operatorname{Vol_h}(B_h(p,r))}{r^n}= \omega_n$ when $r\le 2\tanh \frac{r_0}{2}$. Then we conclude that  $
		\Delta_g d_g(p,x)= (n-1)\coth d_g(p,x)$ in distribution sense in $B_g(p,r_0)$ and $d_g$ is $C^{\infty}$. So we get from (\ref{lowerbound_R_h}) that
		$R_h\equiv 0$ in   $B_h(p,2\tanh \frac{r_0}{2})$. Then  $B_h(p,2\tanh \frac{r_0}{2})$ is flat by
		Theorem \ref{rigidity}.  Hence  $B_g(p,r_0)$ must be isometric to  the $r_0$-ball in hyperbolic space with its sectional curvature equals to $-1$ since $B_g(p,r_0)$ is conformal to $B_h(p,2\tanh \frac{r_0}{2})$ with $h=\left(\cosh{\frac{d_g}{2}}\right)^{-4}g$.
	\end{proof}

		\begin{proof}[Proof of Theorem \ref{rigidity_Sn}]
		(i)
		Firstly, under the conformal change $h=\left(\cos{\frac{d_g}{2}}\right)^{-4}g$
		\begin{equation}\label{R_h_Sn}
			\begin{aligned}
				&R_h(x)\\
				=&	\left(\cos{\frac{d_g(p,x)}{2}}\right)^{4}&\left\{R_g(x)-(n-1)-2(n-1)\Delta_g d_g(p,x)\tan{\frac{d_g(p,x)}{2}}-(n-1)^2\left(\tan{\frac{d_g(p,x)}{2}}\right)^{2}\right\}\\
				\ge&	\left(\cos{\frac{d_g(p,x)}{2}}\right)^{4}&\left\{(n-1)^2-2(n-1)\Delta_g d_g(p,x)\tan{\frac{d_g(p,x)}{2}}-(n-1)^2\left(\tan{\frac{d_g(p,x)}{2}}\right)^{2}\right\}
			\end{aligned}	
		\end{equation}
		for all $x\in V_p\cap \operatorname{inj}_g(p)$
		(cf. P192 in \cite{STbook}). By
		Lemma \ref{Ricci_flat_flat} we  have $R_h(x)\le 0$ for all $x\in V_p\cap \operatorname{inj}_g(p)$ if (\ref{LSI_rigidity_2}) holds. Then  we conclude from (\ref{R_h_Sn}) that
		\begin{equation}\label{Delta_upper_bound}
			\Delta_g d_g(p,x)\ge (n-1)\cot d_g(p,x)
		\end{equation}
	 for all $x\in V_p\cap \operatorname{inj}_g(p)$.
		Denote $\{x^k\}^n_{k=1}$ be the normal geodesic coordinates centered at $p$ with respect to metric $g$. Recall that in the geodesic normal coordinates $\{x^k\}^n_{k=1}$ (see formula (3.4) on p. 211 of \cite{STbook})
	\begin{equation}\label{expansion_g_kl}
		g_{k \ell}(x)=\delta_{k \ell}-\frac{1}{3} R_{k p q \ell} x^p x^q+O\left(d_g(p,x)^3\right)
	\end{equation}	
		$\operatorname{det}\left(g_{k \ell}(x)\right) $ has the following expansion near $p$ (see Lemma 3.4 on p. 210 of \cite{STbook})
		$$
		\begin{aligned}
			\operatorname{det}\left(g_{k \ell}(x)\right) &  =1-\frac{1}{3} R_{ij}(p) x^i x^j +O\left(d_g(p,x)^3\right),
		\end{aligned}
		$$
		where $R_{k p q \ell}$ and $R_{ij}$ denote the curvature tensor and the Ricci curvature of $g$.
		And hence we have
		\begin{equation}\label{expansion_Delta_dg}
			\begin{aligned}
				\left(\Delta_g d_g\right)(x)&=\left.\frac{\partial}{\partial r}\log \left( r^{n-1} \sqrt{\operatorname{det}^S g_{i j}}\right)\right|_{r=d_g(p,x)}\\
				&=	\left.\frac{n-1}{d_g(p,x)}+\frac{1}{2}\frac{\partial}{\partial r}\log \left\{ 1-\frac{1}{3} R_{ij}(p) x^i x^j  +O\left(d_g(p,x)^3\right)\right\}\right|_{r=d_g(p,x)}\\
				&=\frac{n-1}{d_g(p,x)}-\frac{1}{3 d_g(p,x)} R_{ij}(p) x^i x^j+O\left(d_g(p,x)^2\right),
			\end{aligned}
		\end{equation}
		since $ \frac{\partial x^i}{\partial r}=\frac{x^i}{r}$, where $ g^S_{i j}$ are the components of the metric in geodesic spherical coordinates  at $p$ of the geodesic sphere $B_g(p, 1)$.
		Note that
		\begin{equation}\label{expansion_cot}
			\cot{d_g(p,x)}=\frac{1}{d_g(p,x)}-\frac{d_g(p,x)}{3}+O\left(d_g(p,x)^3\right).
		\end{equation}	
		Combing with (\ref{Delta_upper_bound}), (\ref{expansion_Delta_dg}) and (\ref{expansion_cot}), we get that
		\begin{equation}\label{Rc_lower_bound}
			Rc_g(p)\le (n-1)g(p).
		\end{equation}
		Then we have $Rc_g(p)=(n-1)g$ since $R_g(p)\ge n(n-1)$.
		Under the conformal change $h=\left(\cos{\frac{d_g}{2}}\right)^{-4}g$, by the directly computations we get at $p$
		\begin{equation}\label{change_of_sec}
			sec_h(p)=sec_g(p)-1,
		\end{equation}
	and hence $R_h(p)=0$ and $Rc_h(p)=0$. Moreover, since $R_g$ achieves its minimum at $p$, $\nabla R_g(p)=0$ and $\Delta_g R_g(p)\ge 0$. It follows from (\ref{expansion_Delta_dg}) that
	\begin{equation}\label{expansion_R_h11}
		\begin{aligned}
			R_h(x)=&\left(\cos{\frac{d_g(x)}{2}}\right)^{4}\left\{R_g(x)-(n-1)-2(n-1)\Delta_g d_g(p,x)\tan{\frac{d_g(p,x)}{2}}-(n-1)^2\left(\tan{\frac{d_g(p,x)}{2}}\right)^{2}\right\}\\
			=&\frac{1}{2}\nabla_i\nabla_j R_g(p)x_ix_j+\frac{n-1}{3}
			R_{ij}(p) x^i x^j-\frac{(n-1)^2}{3}d_g(x)^2+O\left(d_g(x)^3\right),
		\end{aligned}	
	\end{equation}
where we used the expansion $\tan(y)=y+\frac{y^3}{3}+O(y^5)$.
Notice that the Christoffel symbols satisfies that $(\Gamma_g)^k_{ij}=O\left(d_g(p,x)\right)$ by (\ref{expansion_g_kl}) and also $(\Gamma_h)^k_{ij}=O\left(d_g(p,x)\right)$.	Then we have $\Delta_h =\Delta_g$ at $p$ and by (\ref{expansion_R_h11})
	\begin{equation*}
		\Delta_h R_h(p)=\Delta_g R_h(p)=\sum\limits_{p}\frac{\partial^2 R_h}{\partial\left(x^p\right)^2}(p)=\Delta_gR_g(p)+\frac{2(n-1)}{3}
		R_g(p)-\frac{2n(n-1)^2}{3}=\Delta_gR_g(p)\ge 0.
	\end{equation*}
Then by (\ref{LSI_rigidity_2}), (\ref{heat_kernel_up_best}) and the heat kernel expansion (\ref{heat_kernel_expansion}), we can use the same arguments in Lemma \ref{Ricci_flat_flat} to get the second term of the heat kernel expansion 
$\phi_2(p,p)=\left(\frac{1}{30} \Delta_h R_h+\frac{1}{72} R_h^2-\frac{1}{180}|\mathrm{Rc_h}|^2+\frac{1}{180}|\mathrm{Rm_h}|^2\right)(p)\le 0$. Hence $|\mathrm{Rm_h}|(p)=0$ since we have $R_h(p)=0$, $Rc_h(p)=0$ and $\Delta_h R_h(p)\ge0$. Then Theorem \ref{rigidity_Sn} (i) follows from (\ref{change_of_sec}).
		
		(ii) Note that we have $R_h(x)\le 0$ for all $x\in V_p\cap \operatorname{inj}_g(p)$ if (\ref{LSI_rigidity_2}) holds by Lemma \ref{Ricci_flat_flat}. Then by the same arguments as in Theorem \ref{rigidity_Iso_Hn}, we have 
		\begin{equation}\label{_Delta_upper}
			\Delta_g d_g(p,x)\le (n-1)\coth d_g(p,x)
		\end{equation}
		for all $x\in V_p\cap \operatorname{inj}_g(p)$. Under the conformal change $h=\left(\cosh{\frac{d_g}{2}}\right)^{-4}g$, by the directly computations we get at $p$
		\begin{equation}\label{change_orf_sec}
			sec_h(p)=sec_g(p)+1.
		\end{equation}
		Hence $R_g(p)=R_h(p)-n(n-1)\le -n(n-1)$ and  $R_g(p)=-n(n-1)$ by (\ref{positive_scalar_cuvature11}). Moreover, we still have $\nabla R_g(p)=0$ and $\Delta_g R_g(p)\ge 0$ since $R_g$ achieves its minimum at $p$. 	Note that
		\begin{equation}\label{expansion_coth}
			\coth{d_g(p,x)}=\frac{1}{d_g(p,x)}+\frac{d_g(p,x)}{3}+O\left(d_g(p,x)^3\right).
		\end{equation}	
	Combining this with (\ref{expansion_Delta_dg}) and (\ref{_Delta_upper}), we conclude that  $Rc_g(p)\ge -(n-1)g(p)$ and hence $Rc_g(p)= -(n-1)g(p)$.  It follows from (\ref{expansion_Delta_dg}) and (\ref{lowerbound_R_h}) that
	\begin{equation}\label{expansion_R_h111}
		\begin{aligned}
			R_h(x)
			=&\frac{1}{2}\nabla_i\nabla_j R_g(p)x_ix_j-\frac{n-1}{3}
			R_{ij}(p) x^i x^j-\frac{(n-1)^2}{3}d_g(x)^2+O\left(d_g(x)^3\right),
		\end{aligned}	
	\end{equation}
	where we used the expansion $\tanh(y)=y-\frac{y^3}{3}+O(y^5)$.
	  The rest of the proof of Theorem \ref{rigidity_Sn} (ii) is similar to Theorem \ref{rigidity_Sn} (i).
	\end{proof}

	\textbf{Acknowledgement}: We would like to thank Prof. Gilles Carron for useful comments of our paper.

\end{document}